\documentclass[12pt]{article}

\usepackage{amsmath}
\usepackage{amssymb,longtable}
\usepackage[russian,english]{babel}
\usepackage{amsthm}

\newtheorem{definition}{Definition}

\newtheorem{sled}{Corollary}
\newtheorem{hypothesis}{Conjecture}
\newtheorem*{hypothesis*}{Conjecture 1$'$}

\newtheorem{lemma}{Lemma}
\newtheorem{theor}{Theorem}
\newtheorem*{theor*}{Theorem}
\newtheorem{proposition}{Proposition}

\newcommand{\sign}{\operatorname{sign}}
\newcommand{\intt}{\operatorname{lt}}

\theoremstyle{definition}

\begin{document}

\author{A.\,A.~Gornitskii\thanks{The work was supported by RFBR grant 16-01-00818.}}

\title {Essential Signatures and Monomial Bases for $B_n$ and $D_n$.}
\date{}

\maketitle

\thispagestyle{empty}
\begin{abstract}
In the representation theory of simple Lie algebras, we consider the problem of constructing a monomial basis in an arbitrary irreducible finite-dimensional highest weight module. We construct a PBW-type basis in every finite-dimensional representation of $B_n$ and $D_n$ and we describe the associated semigroup of essential signatures. These bases are parameterized by integer points in some polytopes. We give the inequalities defining these polytopes.
\end{abstract}
\section {Introduction}

Let $\mathfrak{g}$ be a simple complex Lie algebra.  One has the triangular decomposition
$\mathfrak{g}=\mathfrak{u}^{-}$$\oplus$$\mathfrak{t}$$\oplus$$\mathfrak{u}$, where
$\mathfrak{u}^{-}$ and $\mathfrak{u}$ are mutually opposite maximal unipotent subalgebras and $\mathfrak{t}$ is a Cartan subalgebra.

One has: $\mathfrak{u}= \langle e_{\alpha}$ $\mid$ $\alpha$ $\in$ $\Delta_{+}\rangle$,
$\mathfrak{u^{-}}=\langle e_{-\alpha}$ $\mid$ $\alpha$ $\in$ $\Delta_{+}\rangle$, where
$\Delta_{+}$ is the system of positive roots, $e_{\pm\alpha}$
 are the root vectors, and the notation $\langle\ldots\rangle$ stands for the linear span.

 We denote a finite-dimensional irreducible $\mathfrak{g}$-module with highest weight $\lambda$ by $V(\lambda)$ and a highest weight vector in this module by $v_{\lambda}$. By a \emph{monomial} basis in $V(\lambda)$ we mean a basis consisting of vectors of the form $e_{-\alpha_{i_1}}\ldots e_{-\alpha_{i_k}}\cdot v_{\lambda},$ where $\alpha_{i}\in \Delta_{+}$, up to a multiple. By a \emph{PBW-type} basis (PBW stands for Poincare-Birkhoff-Witt) in $V(\lambda)$ we mean a monomial basis consisting  of vectors of the form $e_{-\alpha_{1}}^{p_1}\ldots e_{-\alpha_{N}}^{p_N}\cdot v_{\lambda},$ $p_i\in\mathbb{Z}_{\geq0}$, for some fixed numeration of positive roots $\{\alpha_1,\ldots,\alpha_N\}$ up to a multiple.

  Various approaches to construction of monomial bases in $V(\lambda)$ are known. For example, for every reduced decomposition of the longest element in the Weyl group one may obtain a basis in each $V(\lambda)$ by applying the lowering operators corresponding to simple roots to the highest weight vector (see \cite{[L]}). These bases, which we call \emph{string bases}, are parametrized by integer points in the so-called string cone. Another way to obtain a monomial basis is coming from Lusztig's parametrization of the canonical basis in the quantized enveloping algebra (see \cite{[Lu]},\cite{[Li]}). Every reduced decomposition of the longest element in the Weyl group defines a numeration of positive roots $\{\alpha_1,\ldots,\alpha_N\}$. The vectors $\frac{e_{-\alpha_{1}}^{p_{{1}}}}{p_1!}\cdot \ldots \cdot
 \frac{e_{-\alpha_{N}}^{p_{{N}}}}{p_N!}\cdot{v}_{\lambda}$ generate $V(\lambda)$ and one chooses a basis from this set of vectors by using some weighted opposite right lexicographic order on $\mathbb{Z}^{N}$. We call this basis \emph{Lusztig's PBW-type basis}. Also, one may obtain a PBW-type basis corresponding to an arbitrary numeration of positive roots and a homogeneous order on $\mathbb{Z}^{N}$ by using FFLV (Feigin-Fourier-Littelmann-Vinberg) approach (see \cite{[V]},\cite{[FFL1]},\cite{[FFL2]}).

  We construct (see Section \ref{orthcase}) a PBW-type basis for some non-homogeneous order on $\mathbb{Z}^{N}$ in $V(\lambda)$ for all irreducible representations of the Lie algebra $\mathfrak{g}$ of type $B_n$ or $D_n$. These bases do not coincide with Lusztig's PBW-type bases: at least the numeration of positive roots is different. However, our bases are close to Lusztig's PBW-type bases in the following sense. Our approach can be used for constructing the PBW-type bases for Lie algebras $A_n$ and $C_n$ as well. In these cases our bases coincide with Lusztig's bases for some reduced expression of the longest element in the Weyl group. We are grateful to the referee for this observation.

   The basic concept is defined as follows.

\begin{definition}
 A signature is an $(N+1)$-tuple
 $\sigma=(\lambda;p_{{1}},\dots,p_{{N}})$,
  where $N$ is the number of positive roots numbered in a certain fixed order: $\Delta_{+}=\{\alpha_{1},\dots,\alpha_{N}\}$, $\lambda$ is a dominant weight, and $p_{i}\in\mathbb{Z}_{+}$.
\end{definition}

Set
$$
{v}(\sigma)=\frac{e_{-\alpha_{1}}^{p_{{1}}}}{p_1!}\cdot \ldots \cdot
 \frac{e_{-\alpha_{N}}^{p_{{N}}}}{p_N!}\cdot{v}_{\lambda}.
$$
$\newline$
$\lambda$ is called the \emph{highest weight} of $\sigma$, the eigenweight $\lambda-\sum p_i\alpha_i$ of the vector $v(\sigma)$ is called the \emph{weight} of $\sigma$, and the numbers $(p_1,\ldots,p_N)$ are called the \emph{exponents} of $\sigma$. Thus we have defined a vector in $V(\lambda)$ for every signature with highest weight $\lambda$. The vectors ${v}(\sigma)$ generate $V(\lambda)$, but they are linearly dependent. Our goal is to select a basis of $V(\lambda)$ from the set of all vectors ${v}(\sigma)$.

Now we explain an approach to solving this problem. To this end, we need to equip the set of signatures with an order. Let $\omega_1, \ldots, \omega_n$ be the fundamental weights and let $$\sigma=(\lambda;p_1,\ldots,p_N),\quad \sigma'=(\lambda';p_1',\ldots,p_N'),$$
$$ \lambda=\sum k_{i}\omega_{i},\quad \lambda'=\sum k_{i}'\omega_{i},\;\qquad k_{i},k_{i}'\in\mathbb{Z}_{+}.$$
First we compare the tuples $(k_1,\ldots,k_n)$ and $(k_1',\ldots,k_n')$ by using the degree lexicographic order and put $\sigma<\sigma'$ if $(k_1,\ldots,k_n)<(k'_1,\ldots,k'_n)$.
If $\lambda=\lambda'$, then we compare the tuples $(p_1,\ldots, p_N)$ and $(p'_1,\ldots, p'_N)$ by using a fixed monomial order on $\mathbb{Z}_{\ge0}^{N}$.

 \begin{definition} A signature $\sigma$ is essential, if
 $v(\sigma)\notin\langle v(\tau)\mid\tau<\sigma\rangle$.

\end{definition}
The following statement is obvious.
\begin{proposition}
The set $\{{v}(\sigma)\mid\sigma \mbox{ essential}\}$ is a basis of $V(\lambda)$.
\end{proposition}

 The essential signatures with given highest weight $\lambda$ parametrize the desired monomial basis of $V(\lambda)$.

 Let $\mathfrak{t}_{\mathbb{Z}}\subset\mathfrak{t}$ be the coroot lattice, i.e., the lattice of vectors on which all weights take integer values. The following proposition was proved by Vinberg. For convenience of
the reader, we provide a proof in Section \ref{2}.

 \begin{proposition}
 The essential\label{1} signatures form a semigroup $\Sigma=\Sigma_{\mathfrak{g}}$ in $\mathfrak{t}^{*}_{\mathbb{Z}}\oplus\mathbb{Z}^{N}$.

\end{proposition}
Thus monomial bases of the above type are described by the semigroup of essential signatures $\Sigma$. These bases (including Lusztig's PBW-type bases) and the string bases are special cases of bases coming from the method of birational sequences (see \cite{[Li]}).

Vinberg formulated some conjectures about the structure of the set of essential signatures. (In fact, all these conjectures were formulated for a \emph{homogeneous} order on $\mathbb{Z}^{N}$.) Here is the first conjecture of Vinberg:
\begin{hypothesis} There exist a numeration of positive roots and a monomial order on $\mathbb{Z}^{N}$ such that the semigroup $\Sigma$ is generated by the essential signatures of fundamental highest weights.
\label{conjecture1}
 \end{hypothesis}

A weaker version of
this conjecture can be formulated for a prescribed set of dominant
weights $\{\lambda_1,\ldots,\lambda_m\}$ containing all the fundamental weights:
 \begin{hypothesis*}
The semigroup $\Sigma$ (for the chosen numeration of positive roots and the monomial order on signatures) is generated by the essential signatures of highest weights in $\{\lambda_1,\ldots,\lambda_m\}$.
 \end{hypothesis*}

Let us formulate the other conjectures of Vinberg. Let $\Sigma_{\mathbb{Q}}$ be the rational cone spanned by $\Sigma$. Then this cone can be defined by linear inequalities. (The number of these inequalities is finite if Conjecture \ref{conjecture1}$'$ holds.)
\begin{hypothesis}
 The semigroup $\Sigma$ is saturated, i.e.,
 $\Sigma=\Sigma_{\mathbb{Q}}\bigcap (\mathfrak{t}_{\mathbb{Z}}^{*}\oplus
 \mathbb{Z}^{N})$.
 \label{conjecture2}
\end{hypothesis}
Conjecture \ref{conjecture2} claims that the bases of $V(\lambda)$ are parametrized by lattice points of flat sections of some polyhedral cone.

  It is known (see \cite{[Li]}) that for Lusztig's and string bases the semigroup $\Sigma$ is finitely generated and Conjecture \ref{conjecture2} holds. But in these cases it is difficult to give an explicit description for the semigroup $\Sigma$, in particular, the generators of $\Sigma$ are not known.

The next conjecture refines the structure of the polyhedral cone in Conjecture \ref{conjecture2}.
\begin{hypothesis}
\label{conjecture3}
 There exist a family of subsets $M_{i}\subset
  \{1,\dots,N\}$ and a family of elements $l_{i}\in \mathfrak{t}_{\mathbb{Z}}$ such that the set of essential signatures $\sigma=(\lambda;p_{1},\dots,p_{N})$ of highest weight $\lambda$ is given by the inequalities

  $$
  \sum_{j\in M_{i}}p_{j}\leq \lambda(l_{i}).
  $$

\end{hypothesis}

Conjectures \ref{conjecture1}, \ref{conjecture2}, \ref{conjecture3} (for a certain homogeneous order on signatures) were proved for Lie algebras of types $A_n$, $C_n$ and $G_2$ (\cite{[FFL1]}, \cite{[FFL2]}, \cite{[G]}). It looks plausible that Conjecture \ref{conjecture3} is wrong in other cases.

Conjectures \ref{conjecture1}, \ref{conjecture2} were also proved for certain homogeneous orders in cases $B_3$, $D_4$ (\cite{[T]}, \cite{[G1]}). Moreover, inequalities defining $\Sigma$ were described.

 Finally, the semigroup $\Sigma$ was described explicitly in case $B_n$ for some numeration of positive roots and monomial order such that Conjectures \ref{conjecture1}$'$ (for some set of dominant weights) and \ref{conjecture2} are true (\cite{[Ì]}).

The main result of the article is the following theorem (see Theorem \ref{main}):
\begin{theor*}
 The following statements are true for $D_n$, $B_n$ ($n\geq2$):

For some fixed numeration of positive roots and monomial order on signatures, the semigroup $\Sigma$ is saturated and generated by essential signatures of highest weights in the set $$\{\omega_1,\ldots,\omega_n,2\omega_{n-1},2\omega_n,\omega_{n-1}+\omega_n\}\quad\textrm{for $D_n$}$$
and in the set
 $$\{\omega_1,\ldots,\omega_n,2\omega_{n}\}\quad\textrm{for $B_n$, respectively}. $$

\end{theor*}
Here the numeration of fundamental weights is the standard one, described in Section \ref{orthcase}. Thus Conjectures \ref{conjecture1}$'$ and 2 are true for $B_n$ and $D_n$ with respect to the chosen numeration and (not homogeneous) monomial order. Moreover,
we find inequalities defining the cone $\Sigma_{\mathbb{Q}}$ for $B_n$ and $D_n$ (Theorems~\ref{ineqbn} and \ref{ineqdn}, respectively).
 All these results were obtained independently from \cite{[Ì]} by using a completely different method.

Let us briefly describe the structure of the paper. In Section \ref{Conj1} we give a necessary and sufficient condition for Conjecture~\ref{conjecture1} or Conjecture~\ref{conjecture1}$'$ to be true, and we explain how this condition can be verified. In the rest of the article we construct monomial bases for $B_n$ and $D_n$ by using an inductive procedure, starting from $D_2=A_1+A_1$, and prove the above theorem.
\section*{Acknowledgement}
We are grateful to the referee for his careful reading and insightful
 comments and suggestions.

\section{The semigroup of essential signatures}
\label{2}

Here we show, following an argument of Vinberg, that the essential signatures of all highest weights form a semigroup.

Let $G$ be a simply connected simple complex algebraic group such that $\mathop{\mathrm{Lie}}G=\mathfrak{g}$. Let $T$ be the maximal torus in $G$ such that $\mathop{\mathrm{Lie}}T=\mathfrak{t}$ and $U$ be the maximal unipotent subgroup of $G$ such that
 $\mathop{\mathrm{Lie}}U=\mathfrak{u}$.
 Consider the homogeneous space $G/U$.
 Let $B=T\rightthreetimes U$ be the Borel subgroup. Then

$$
 \mathbb{C}[G/U]=\bigoplus_{\lambda} \mathbb{C}[G]_{\lambda}^{(B)},
$$
where $$\mathbb{C}[G]_{\lambda}^{(B)}=\{f\in\mathbb{C}[G]\mid f(gtu)=\lambda(t)f(g),\, \forall g\in G, t\in T, u\in U\}$$
is the subspace of eigenfunctions of weight $\lambda$ for $B$ acting on $\mathbb{C}[G]$ by right translations of an argument.
 Each subspace $\mathbb{C}[G]_{\lambda}^{(B)}$ is finite-dimensional and is isomorphic as a $G$-module (with respect to the action of $G$ by left translations of an argument) to the space $V(\lambda)^{*}$ of linear functions on $V(\lambda)$ (see, e.g., \cite{[Ï]}, Theorem 3). The isomorphism is given by the formula:
$$
 V(\lambda)^{*}\ni\omega \longmapsto f_{\omega}\in\mathbb{C}[G]_{\lambda}^{(B)},
 \quad\textrm{where}\quad f_{\omega}(g)=\langle\omega, g\emph{v}_{\lambda}\rangle.
$$

Let $U^{-}$ be the maximal unipotent subgroup such that $\mathop{\mathrm{Lie}}U^{-}=\mathfrak{u^{-}}$. The function $f_{\omega}$ is uniquely determined by its restriction to the dense open subset $U^{-}\cdot$$T\cdot$$U$; moreover
\begin{multline*}
 $$f_{\omega}(u^{-}\cdot t\cdot u)=\langle\omega
,u^{-} tu\emph{v}_{\lambda}\rangle=\langle\omega
,\lambda(t)u^{-}\emph{v}_{\lambda}\rangle=\lambda(t)f_{\omega}(u^{-}),\\ \quad \forall u\in U, u^{-}\in U^{-}, t\in T.
$$
\end{multline*}
Next, $U^{-}=U_{-\alpha_{1}}\cdot\ldots\cdot U_{-\alpha_{N}}$, where
$U_{\alpha}=\{\exp(ze_{\alpha})\mid$ $z\in\mathbb{C}\}$ (see \cite[Sec. X, \S 28.1]{[X]}). Hence

$$
 u^{-}=\exp(z_{1}e_{-\alpha_{1}})\cdot\ldots\cdot\exp(z_{N}e_{-\alpha_{N}}).
$$
Thus we obtain
$$
f_{\omega}(u^{-})=\left\langle\omega,
\exp(z_{1}e_{-\alpha_{1}})\cdot\ldots\cdot\exp(z_{N}e_{-\alpha_{N}})
\cdot \emph{v}_{\lambda}\right\rangle=\sum_{\sigma=(\lambda;p_{1},\dots,p_{N})}
\prod_{i=1}^{N} z_{i}^{p_{i}}
\langle\omega,\emph{v}(\sigma)\rangle.
$$

\begin{proposition}
 A signature $\sigma$ is essential if and only if
$\prod z_{i}^{p_{i}}$ is the least term in
 $f_{\omega}|_{U^{-}}$ for some
$\omega \in V(\lambda)^{*}$ in the sense of the chosen monomial order on $\mathbb{Z}_{\geq0}^{N}$.
\end{proposition}
\begin{proof}
Let $\prod z_{i}^{p_{i}}$ be the least term in $f_{\omega}|_{U^{-}}$ for some $\omega\in V(\lambda)^{*}$. Then $\omega$ vanishes on all vectors $\emph{v}(\tau)$ with $\tau<\sigma$ and is nonzero at  $\emph{v}(\sigma)$. Consequently, $\emph{v}(\sigma)$ cannot be expressed via $\emph{v}(\tau)$ with $\tau<\sigma$, and hence  $\sigma$ is essential.

Conversely, let $\sigma$ be essential. Consider a function $\omega$ that vanishes on $\emph{v}(\tau)$ for all essential $\tau$ except  $\sigma$. Obviously, $f_{\omega}|_{U^{-}}$ has the desired least term.
\end{proof}

\begin{proof}[Proof of Proposition \ref{1}]
Suppose that the least terms in $f|_{U^{-}}$ and $g|_{U^{-}}$ correspond to the essential signatures $\sigma$ and $\pi$ of highest weights $\lambda$ and $\mu$, respectively. Then the least term in $(f\cdot g)|_{U^{-}}$ corresponds to the signature $\sigma+\pi$ of highest weight $\lambda+\mu$. Hence $\sigma+\pi$ is essential.
\end{proof}

\section {Conjecture 1$'$}
\label{Conj1}
Fix some dominant weights $\lambda_1,\ldots,\lambda_m$. Let $S$ be the semigroup of dominant weights generated by $\lambda_1,\ldots,\lambda_m$.
Denote by $\Sigma^{(\lambda_1,\ldots,\lambda_m)}=\Sigma_{\mathfrak{g}}^{(\lambda_1,\ldots,\lambda_m)}$ the semigroup of essential signatures generated by the essential signatures of highest weights in $\{\lambda_1,\ldots,\lambda_m\}$; $\Sigma(S)$ stands for the semigroup of essential signatures of highest weights in $S$.

Here we give a necessary and sufficient condition for $\Sigma(S)=\Sigma^{(\lambda_1,\ldots,\lambda_m)}$.

It is known that the algebra $\mathbb{C}[G/U]$ is generated by the subspaces $\mathbb{C}[G/U]^{(B)}_{\omega_i}$ $(i=1,\ldots,n)$.
 Set $X=\mathop{\mathrm{Spec}}\bigoplus_{\lambda\in S}\mathbb{C}[G/U]^{(B)}_{\lambda}$. Then (see \cite{[V-Ï]})
 $$X\simeq\overline{G(v_{\lambda_1}+\ldots+v_{\lambda_m})}\subseteq V(\lambda_1)\oplus\ldots\oplus V(\lambda_m)=V.$$
  Let $I$ be the vanishing ideal of $X$ in $\mathbb{C}[V]$, and
  $x_{i}$ be the coordinates on $V$ corresponding to the basis $\{v(\sigma_i)\mid \sigma_i \textrm{ essential}\}$. So we have a surjective algebra homomorphism
  $$\phi:\quad\mathbb{C}[V]\rightarrow\mathbb{C}[X],
  $$
  such that ker$(\phi)=I$.

  Consider an arbitrary monomial $x_{i_{1}}\ldots x_{i_{k}}$ in $\mathbb{C}[V]$. Set $\sign(x_{i_{1}}\ldots x_{i_{k}})=\sigma_{i_{1}}+\ldots+\sigma_{i_{k}}=\sigma$. We call $\sigma$ $\emph{the signature of the monomial}$ $x_{i_{1}}\ldots x_{i_{k}}$. This defines a grading of $\mathbb{C}[V]$ by $\Sigma^{(\lambda_1,\ldots,\lambda_m)}$. The degree of any polynomial $h\in\mathbb{C}[V]$ homogeneous with respect to this grading is also denoted by $\sign(h)$. One has:
  $$
  \mathbb{C}[V]\xrightarrow{\phi}\mathbb{C}[X]\subset\mathbb{C}[U^{-}\cdot T].
  $$
  We have the grading by the group $\mathfrak{t}_{\mathbb{Z}}^{*}\oplus
 \mathbb{Z}^{N}$ on the algebra $\mathbb{C}[U^{-}\cdot T]$. We also denote by $\sign(h)$ the degree of any homogeneous polynomial $h\in\mathbb{C}[U^{-}\cdot T]$ with respect to this grading.
  For any $f\in\mathbb{C}[V]$ or $f\in\mathbb{C}[U^{-}\cdot T]$, denote by $\intt(f)$ the least term of $f$ with respect to these gradings.

   The map $\phi$ is not a graded algebra homomorphism, but it is compatible with the decreasing filtrations, induced by the corresponding gradings. We have coarser gradings on $\mathbb{C}[V]$, $\mathbb{C}[X]$, and $\mathbb{C}[U^{-}\cdot T]$ by $\mathfrak{t}_{\mathbb{Z}}^{*}$. The map $\phi$ is a graded algebra homomorphism for coarser gradings. Notice, that the homogeneous components of $\mathbb{C}[V]$ and $\mathbb{C}[X]$ are finite-dimensional.

   Let $\intt(I)$ be the ideal spanned by $\intt(f), f\in I$.
  Denote by $J$ the ideal in $\mathbb{C}[V]$ spanned by the binomials
  $$x_{i_{1}}\ldots x_{i_{k}}-x_{j_{1}}\ldots x_{j_{l}} \quad\textrm{such that}\quad \sign(x_{i_{1}}\ldots x_{i_{k}})=\sign(x_{j_{1}}\ldots x_{j_{l}}).$$

 Let $f$ be a regular function on $V$, which can be expressed as a polynomial in $\{x_{i}\}$. Then $\phi(f)$ is a function on $X$, hence a polynomial in $\{z_j\}$, $j=1,\ldots,N$ (coordinates on $U^{-}$) and $\{t_k\}$, $k=1,\ldots,n$ (coordinates on $T$ corresponding to fundamental weights). Thus for any monomial $x_{i_1}\ldots x_{i_{s}}$ of signature $\sigma=(\lambda;p_1,\ldots,p_N)$ the least term of $\phi(x_{i_1}\ldots x_{i_{s}})$ is $z^{\sigma}:=t^{\lambda}\prod z^{p_j}_j$, where $t^{\lambda}:=t^{k_1}_1\ldots t^{k_n}_n$, $\lambda=\sum k_i\omega_i$.
 \begin{lemma}
 $\intt(I)\subset J$.
 \end{lemma}
 \begin{proof}
 Let $f\in I$, then $\phi(f)=0$. Let $\sigma=(\lambda;p_1,\ldots,p_N)=\sign(\intt(f))$. The least term of $\phi(\intt(f))$ is greater than $z^{\sigma}$, because $\phi(f)=0$. Hence $\intt(f)$ is a linear combination of monomials of signature $\sigma$ with the sum of coefficients equal to 0, i.e. $\intt(f)\in J$.
 \end{proof}

\begin{proposition}
$\Sigma(S)=\Sigma^{(\lambda_1,\ldots,\lambda_m)}$ if and only if $\intt(I)=J$.
\label{proposition4}
\end{proposition}
\begin{proof}
 The equality $\Sigma(S)=\Sigma^{(\lambda_1,\ldots,\lambda_m)}$ means that for any polynomial $f\in\mathbb{C}[V]$ one has $\sign(\intt(\phi(f)))=\sign(\intt(\phi(x_{i_{1}}\ldots x_{i_{k}})))$ for some monomial $x_{i_{1}}\ldots x_{i_{k}}$.
 Assume that $\Sigma(S)=\Sigma^{(\lambda_1,\ldots,\lambda_m)}$ and let $f\in J$. We can suppose that $f$ is homogeneous for the signature grading. Let $\sigma=(\lambda;p_1,\ldots,p_N)=\sign(f)$. Then  $\intt(\phi(f))>z^{\sigma}$. There exists a monomial $x_{i_{1}}\ldots x_{i_{k}}$ such that $\intt(\phi(x_{i_{1}}\ldots x_{i_{k}}))=\intt(\phi(f))$. So, for some constant $c$ one has $\intt(\phi(f_1))>\intt(\phi(f))$, where $f_{1}=f-cx_{i_{1}}\ldots x_{i_{k}}$. We can continue this process by subtracting a monomial from $f_s$ ($s=1,2,\ldots$) to obtain $f_{s+1}$ such that $\intt(\phi(f_{s+1}))>\intt(\phi(f_{s}))$. The process will stop, since there exist finitely many essential signatures with fixed highest weight. Hence we obtain $f_{m}$ such that $\phi(f_{m})=0$ and $\intt(f_m)=f$. Thus $f\in \intt(I)$.

Conversely, suppose that $\intt(I)=J$. Let $f\in\mathbb{C}[V]$. We look for a monomial $x_{i_{1}}\ldots x_{i_{k}}$ such that  $\sign(\intt(\phi(x_{i_{1}}\ldots x_{i_{k}})))=\sign(\intt(\phi(f)))$. We can assume that $f$ is a polynomial such that the signature $\sign(\intt(f))$ is maximal over polynomials in $\phi^{-1}\phi(f)$. (Here we use that $I$ is graded by $\mathfrak{t}_{\mathbb{Z}}^{*}$ and the homogeneous components are finite-dimensional.) Let $h=\intt(f)$. If $h\notin J$, then we may take for $x_{i_1}\ldots x_{i_k}$ any monomial of $h$. Otherwise $h\in J$, hence there exist $f'\in I$ such that $h=\intt(f')$. Then $\phi(f-f')=\phi(f)$ and $\sign(\intt(f-f'))>\sign(\intt(f))$. A contradiction.
\end{proof}
Now we want to explain how the condition above can be verified. To prove the equality $\intt(I)=J$ it is enough to verify two properties:
\begin{enumerate}
\item the ideal $J$ is generated by polynomials of degree 2;
\item for any two weights $\lambda_i$, $\lambda_j$ any essential signature of highest weight $\lambda_i+\lambda_j$ belongs to $\Sigma^{(\lambda_1,\ldots,\lambda_m)}$.
\end{enumerate}
Indeed, if the first property holds then to show that $\intt(I)=J$ it is enough to prove that if $f\in J$ is a homogeneous (for the signature grading) polynomial of degree 2 then $f\in \intt(I)$. It follows from the proof of Proposition \ref{proposition4} (now $\sign(f)$ has the highest weight $\lambda_i+\lambda_j$) that it is true if the second property holds.

Now we discuss how the first property can be verified. Let $$\sigma=\sigma_1+\ldots+\sigma_k$$
 be a decomposition of $\sigma\in\Sigma^{(\lambda_1,\ldots,\lambda_m)}$ into a sum of essential signatures of highest weights in $\{\lambda_1,\ldots,\lambda_m\}$. Consider the following operations on the decomposition:
  \begin{enumerate}
  \item replacing a pair of signatures $\sigma_i, \sigma_j$ by a pair of essential signatures $\sigma'_i, \sigma'_j$ of  highest weights in $\{\lambda_1,\ldots,\lambda_m\}$ such that $\sigma_i+\sigma_j=\sigma'_i+\sigma'_j$;
  \item replacing a pair of signatures $\sigma_i, \sigma_j$ by an essential signature $\sigma'_s$ of highest weight in $\{\lambda_1,\ldots,\lambda_m\}$ such that $\sigma_i+\sigma_j=\sigma'_s$;
  \item replacing a signature $\sigma_s$  by a pair of essential signatures $\sigma'_i, \sigma'_j$ of highest weights in $\{\lambda_1,\ldots,\lambda_m\}$ such that $\sigma_s=\sigma'_i+\sigma'_j$.
  \end{enumerate}

   We call such operations \emph{admissible}. Obviously, the ideal $J$ is generated by  binomials of degree 2 if and only if for any $\sigma\in\Sigma^{(\lambda_1,\ldots,\lambda_m)}$ we can obtain a given decomposition $\sigma_1+\ldots+\sigma_k$ of $\sigma$ from any other decomposition $\tau_1+\ldots+\tau_l$ by applying admissible operations.

The second property can be verified as follows. First, we compute the number of signatures of highest weights $\lambda_{i}+\lambda_{j}$ which can be represented as a sum of essential signatures of highest weights in $\{\lambda_1,\ldots,\lambda_m\}$ for all $i,j$. Then we compare this number with $\dim V(\lambda_i+\lambda_j)$, which can be found by Weyl's dimension formula.
We arrive at the following statement:
\begin{theor}
Suppose that the following two properties hold:
\begin{enumerate}
\item[$(*)$] for any signature $\sigma\in\Sigma^{(\lambda_1,\ldots,\lambda_m)}$  we can obtain a given decomposition $\sigma=\sigma_1+\ldots+\sigma_k$ into a sum of essential signatures of highest weights in $\{\lambda_1,\ldots,\lambda_m\}$ from any other such decomposition $\sigma=\tau_1+\ldots+\tau_l$ by applying admissible operations;
\item[$(**)$] for any two weights $\lambda_i$, $\lambda_j$ any essential signature of highest weight $\lambda_i+\lambda_j$ is representable as a sum of essential signatures of highest weights in $\{\lambda_1,\ldots,\lambda_m\}$.
\end{enumerate}
 Then $\Sigma(S)=\Sigma^{(\lambda_1,\ldots,\lambda_m)}$.
 \label{theorem1}
\end{theor}

 Denote by $\Sigma^f$ the semigroup generated by essential signatures of fundamental highest weights. For $\{\lambda_1,\ldots,\lambda_m\}=\{\omega_1,\ldots,\omega_n\}$ the above theorem yields the following corollary:
 \begin{proposition}
 Suppose that the following two properties hold:
 \label{proposition5}
 \begin{enumerate}
\item for any signature $\sigma\in\Sigma^f$ we can obtain a given decomposition $\sigma=\sigma_1+\ldots+\sigma_k$ into a sum of essential signatures of fundamental highest weights from any other such decomposition $\sigma=\tau_1+\ldots+\tau_k$ by applying admissible operations (admissible operations of the first type only are applicable here);
\item for any two fundamental weights $\omega_i$, $\omega_j$ any essential signature of highest weight $\omega_i+\omega_j$ is representable as a sum of essential signatures of highest weights $\omega_i$ and $\omega_j$.
\end{enumerate}
Then $\Sigma=\Sigma^f$.
 \end{proposition}

\section {Orthogonal case}
\label{orthcase}
In this section we prove (see Theorem \ref{main}) that for some fixed numeration of positive roots and monomial order on signatures the semigroup $\Sigma$ is saturated and generated by essential signatures of highest weights $\omega_i$ ($i=1,\ldots,n$) and $2\omega_n$ for $\mathfrak{g}$ of type $B_n$, and by essential signatures of highest weights $\omega_i$ ($i=1,\ldots,n$), $2\omega_{n-1}$, $2\omega_n$, $\omega_{n-1}+\omega_n$ for $\mathfrak{g}$ of type $D_n$, i.e.:
\begin{eqnarray}
\Sigma_{B_n}&=&\Sigma_{B_n}^{(\omega_1,\ldots,\omega_n,2\omega_n)},\label{B_n}\\ \Sigma_{D_n}&=&\Sigma_{D_n}^{(\omega_1,\ldots,\omega_n,\omega_{n-1}+\omega_n, 2\omega_{n-1},2\omega_n)}\label{D_n}.
\end{eqnarray}
Here the numeration of fundamental weights is according to \cite[Table 1]{[ÂÎ]}.
Moreover, we find the inequalities defining the cone $\Sigma_{\mathbb{Q}}$ for $D_n$ and $B_n$ (see Theorems \ref{ineqdn} and \ref{ineqbn}, respectively).
We shorten the notation as follows:
$$
\Sigma_{B_n}^{(\cdot)}:=\Sigma_{B_n}^{(\omega_1,\ldots,\omega_n,2\omega_n)},\quad \Sigma_{D_n}^{(\cdot)}:=\Sigma_{D_n}^{(\omega_1,\ldots,\omega_n,2\omega_{n-1},2\omega_n,\omega_{n-1}+\omega_n)},
$$
$$ \Sigma'_{D_n}:=\Sigma_{D_n}^{(\omega_1,\ldots,\omega_{n-2},\omega_{n-1}+\omega_n,\omega_{n-1},2\omega_{n-1})},\quad\Sigma''_{D_n}:=\Sigma_{D_n}^{(\omega_1,\ldots,\omega_{n-2},\omega_{n-1}+\omega_n,\omega_{n},2\omega_{n})}.$$

 First we prove that $\Sigma$ is saturated and (\ref{D_n}) holds for $D_2$. Then our strategy will be to show inductively (see the induction hypothesis below) that if $\Sigma$ is saturated and (\ref{D_n}) is true for $D_n$ then $\Sigma$ is saturated and (\ref{D_n}) (resp. (\ref{B_n})) is true for $D_{n+1}$  (resp. $B_n$).

We introduce some notation and recall basic facts about representations of orthogonal Lie algebras.

 Let $\widehat\omega_p=\omega_p$ if $p\neq n-1$ and $\widehat\omega_{n-1}=\omega_{n-1}+\omega_{n}$ for $D_n$, and let
$\widehat\omega_p=\omega_p$ if $p\neq n$ and $\widehat\omega_{n}=2\omega_{n}$ for $B_n$.

Recall that $V(\omega_1)$ is the standard representation of $\mathfrak{so}_{2n+1}$ (resp. $\mathfrak{so}_{2n}$) in $\mathbb{C}^{2n+1}$ (resp. $\mathbb{C}^{2n}$).

Let $\pm\varepsilon_i$ ($i=1,\ldots, n$) be the nonzero weights of the representation $V(\omega_1)$ of $D_n$ or $B_n$. Then the positive roots of $D_n$ are
$$
\varepsilon_i\pm\varepsilon_j, \quad i<j,\quad i,j\in\{1,\ldots,n\},
$$
and the positive roots of $B_n$ are
$$
\varepsilon_i\pm\varepsilon_j, \quad i<j, \quad i,j\in\{1,\ldots,n\},
$$
$$
\varepsilon_i,\quad i\in\{1,\ldots,n\}.
$$
The fundamental weights and weights $\widehat\omega_i$ can be expressed via $\varepsilon_i$ as follows:
$$
\widehat\omega_i=\varepsilon_1+\ldots+\varepsilon_i,\quad i=1,\ldots,n \quad\textrm{for}\quad B_n,\quad i=1,\ldots,n-1\quad\textrm{for}\quad D_n;
$$

$$
\omega_n=\frac{1}{2}(\varepsilon_1+\ldots+\varepsilon_n)\quad\textrm{for both $B_n$ and $D_n$};
$$
$$
\omega_{n-1}=\frac{1}{2}(\varepsilon_1+\ldots+\varepsilon_{n-1}-\varepsilon_n)\quad\textrm{for}\quad D_n.
$$

Denote by $e_{\pm i}$ eigenvectors in $V(\omega_1)$ of eigenvalues $\pm\varepsilon_i$, and denote by $e_0$ an eigenvector of eigenvalue 0 (for $B_n$).

Denote by $\Sigma_X(\lambda)$ the set of essential signatures of highest weight $\lambda$ for the simple Lie algebra of type $X$. For any signature $\sigma$ denote by $\overline{\sigma}$ the tuple of its exponents, set $\overline{\Sigma}_X(\lambda)=\{\overline\sigma\mid\sigma\in\Sigma_X(\lambda)\}$. We denote by $\overline\alpha_i$ the tuple of exponents with coordinate 1 corresponding to a root $\alpha_i$ and with all other coordinates equal to 0. Let $V_{X}(\lambda)$ be the irreducible representation of the Lie algebra of type $X$ with the highest weight $\lambda$.

 One has:
$$
V_{B_n}(\widehat\omega_p)={\bigwedge}^{p}\mathbb{C}^{2n+1},\quad v_{\widehat\omega_p}=e_1\wedge\ldots\wedge e_p,\quad p=1,\ldots,n,
$$
$$
V_{D_n}(\widehat\omega_p)={\bigwedge}^{p}\mathbb{C}^{2n},\quad v_{\widehat\omega_p}=e_1\wedge\ldots\wedge e_p, \quad p<n,
$$
\begin{multline*}
$$V_{D_n}(2\omega_{n-1})\oplus V_{D_n}(2\omega_{n})={\bigwedge}^n\mathbb{C}^{2n},\quad v_{2\omega_{n-1}} = e_1\wedge\ldots\wedge e_{n-1}\wedge e_{-n},\quad\\ v_{2\omega_{n}} = e_1\wedge\ldots\wedge e_{n-1}\wedge e_{n}.$$
\end{multline*}

The representation $V_{B_n}(\omega_n)$ decomposes into a sum of $2^n$ one-dimensional weight subspaces of weights $\frac{1}{2}(\pm\varepsilon_1\pm\ldots\pm\varepsilon_n)$. The representation $V_{D_n}(\omega_{n-1})$ ($V_{D_n}(\omega_{n})$) decomposes into a sum of $2^{n-1}$ one-dimensional weight subspaces of weights $\frac{1}{2}(\pm\varepsilon_1\pm\ldots\pm\varepsilon_n)$ with odd (resp. even) number of minuses.

Now we formulate the \emph{induction hypothesis} for $D_n$:
\begin{itemize}
\item the properties ($*$), ($**$) (see Theorem \ref{theorem1}) hold for $\Sigma'_{D_n}$ and $\Sigma''_{D_n}$,
\item the semigroup $\Sigma_{D_n}$ is saturated,
\item the following property holds: $$\Sigma'_{D_n}\cup\Sigma''_{D_n}=\Sigma_{D_n}^{(\cdot)}.\eqno{(\dagger)}$$
\end{itemize}
The decomposition ($\dagger$) is a technical property, which we use in the proofs.
\subsection{$D_2$}
The base case in the induction procedure is $n=2$. So we start with $D_2=A_1+A_1$.

\label{b3}
 Let $\beta_1, \beta_2$ be the  simple roots for $D_2$. One has:

$$\beta_1=\varepsilon_1-\varepsilon_2,\quad
\beta_2=\varepsilon_1+\varepsilon_2.
$$
$$
\omega_1=\frac{1}{2}(\varepsilon_1-\varepsilon_2),
\quad
\omega_2=\frac{1}{2}(\varepsilon_1+\varepsilon_2).$$

Let us enumerate the positive roots of $D_2$ as follows:
\begin{center}
\begin{tabular}{ll}
 $\alpha_1=\varepsilon_1-\varepsilon_2$, &
 $\alpha_2=\varepsilon_1+\varepsilon_2.$

\end{tabular}
\end{center}

A monomial order on signatures $\sigma=(\lambda;p_1,p_2)$ of fixed highest weight $\lambda$ is given by the lexicographic order on tuples $(p_2,p_1)$.

 Here are all essential signatures of highest weight $\omega_1$ (the highest weight component is omitted):

\begin{center}
\begin{tabular}{ll}
1. $(0,0)$ &
2. $(1,0)$.
\end{tabular}
\end{center}
Here are all essential signatures of highest weight $\omega_2$ (the highest weight component is omitted):
\begin{center}
\begin{tabular}{ll}
1. $(0,0)$ &
2. $(0,1)$.
\end{tabular}
\end{center}

\begin{lemma}
The induction hypothesis holds for $D_2$.
\end{lemma}
\begin{proof}
Fix some dominant weight $\lambda=k\omega_1+l\omega_2$.
Since any irreducible representation of $D_2=A_1+A_1$ is a tensor product of irreducible representations of $A_1$, one has:
$$
\dim V_{D_2}(k\omega_1+l\omega_2)=\dim V_{A_1}(k\omega)\cdot\dim V_{A_1}(l\omega)=(k+1)\cdot(l+1),
$$
where we denote by $\omega$ the fundamental weight of $A_1$.

 Obviously, any signature $(k\omega_1+l\omega_2;p_1,p_2)$, where $p_1=0,\ldots,k,$ $p_2=0,\ldots,l$, is representable as a sum of  essential signatures of highest weights $\omega_1$ and $\omega_2$. The number of such signatures is exactly $(k+1)\cdot(l+1)$. This implies that $\Sigma_{D_2}=\Sigma^{f}$. Moreover, any essential signature has a unique representation as a sum of essential signatures of fundamental highest weights. The semigroup $\Sigma_{D_2}$ is given by the inequalities:
 \begin{enumerate}
 \item $k\geq0$,
 \item $l\geq0$,
 \item $p_1\leq k$,
 \item $p_2\leq l$.
 \end{enumerate}

 The arguments above show that properties ($*$), ($**$), ($\dagger$) hold and $\Sigma_{D_2}$ is saturated.
\end{proof}

 So we can start an inductive procedure from $D_2$.
\subsection {From $D_n$ to $B_n$}

In this subsection we assume inductively that we have some numeration of positive roots and some monomial order for $D_n$ such that inductive hypothesis for $D_n$ holds.

  Here we prove that $\Sigma_{B_n}$ is saturated and the properties ($*$), ($**$) hold for $B_n$ with
$$\{\lambda_1,\ldots,\lambda_{n+1}\}=\{\omega_1,\ldots,\omega_n,2\omega_n\},$$
where we denote the fundamental weights for $B_n$ by the same letters as for $D_n$, by abuse of notation.

We make an induction step by proving the inductive hypothesis for $D_{n+1}$ in the next subsection.

We have the standard embedding of $SO_{2n}$ in $SO_{2n+1}$ such that the following $D_n$-module decomposition holds:
$$
V_{B_n}(\omega_1)=V_{D_n}(\omega_1)\oplus\langle e_0\rangle.
$$
Since $D_n$ and $B_n$ share the same Cartan subalgebra, we can consider any root of $D_n$ as a root of $B_n$.

First of all we need to extend the monomial order and the numeration of positive roots from $D_n$ to $B_n$. We enumerate short positive roots of $B_n$ after long positive roots, which are positive roots of $D_n$:
$$
\textrm{positive roots of $D_n$},\varepsilon_1,\ldots,\varepsilon_n.
$$

 We extend the monomial order (on signatures with the same highest weight) as follows:
$$
(\overline{\sigma},k_1,\ldots,k_n)< (\overline\sigma',k_1',\ldots,k_n'),
$$
if either $(k_1,\ldots,k_n)<(k_1',\ldots,k_n')$ in degree  lexicographic order or $(k_1,\ldots,k_n)=(k_1',\ldots,k_n')$ and $\overline\sigma<\overline\sigma'$ with respect to the monomial order for $D_n$.

Now we describe the set of essential signatures of highest weights $\lambda_i$ for $B_n$.  For any signature $\sigma$ of $D_n$ we consider $\overline\sigma$ as the tuple $(\overline\sigma,0,\ldots,0)$ for $B_n$.

\begin{lemma}
\label{lemma}
$$
\overline\Sigma_{B_n}(\omega_p)=\overline\Sigma_{D_n}(\widehat\omega_p)\sqcup(\overline\Sigma_{D_n}(\omega_{p-1})+\overline\varepsilon_p),\quad p=1,\ldots,n,
$$
$$
\overline\Sigma_{B_n}(2\omega_{n})=(\overline\Sigma_{D_n}(\widehat\omega_{n-1})+\overline\varepsilon_n)\sqcup(\overline\Sigma_{D_n}(2\omega_{n-1})+2\overline\varepsilon_{n})\sqcup\overline\Sigma_{D_n}(2\omega_{n}),
$$
\end{lemma}

where $\overline\Sigma_{D_n}(\omega_0):=\{(0,\ldots,0)\}$.

\begin{proof}
For any $p<n$ one has the $D_n$-module decomposition:
$$
V_{B_n}(\omega_p)={\bigwedge}^p\mathbb{C}^{2n+1}={\bigwedge}^p\mathbb{C}^{2n}\oplus e_0\wedge{\bigwedge}^{p-1}\mathbb{C}^{2n}=
$$
$$=V_{D_n}(\widehat\omega_p)\oplus e_0\wedge V_{D_n}(\omega_{p-1}),
$$
where the first summand is spanned by the vectors $v(\sigma), \overline\sigma\in\overline\Sigma_{D_n}(\widehat\omega_p)$, and the second summand is spanned by $ v(\sigma),\overline\sigma\in\overline\Sigma_{D_n}(\omega_{p-1})+\overline\varepsilon_p$.
 The signatures $\sigma$ with  $\overline\sigma\in\overline\Sigma_{D_n}(\widehat\omega_p)$ are essential for $D_n$ and have zero exponents corresponding to roots $\varepsilon_i$, hence they are essential for the extended monomial order. The signatures $\sigma$ with  $\overline\sigma\in\overline\Sigma_{D_n}(\omega_{p-1})+\overline\varepsilon_p$ are minimal among those for which $v(\sigma)$ span $V_{B_n}(\omega_p)$ modulo $V_{D_n}(\widehat\omega_p)$, hence they are essential, too. This proves the first equality for $p<n$.

Now we want to prove the equality:
$$
\overline\Sigma_{B_n}(\omega_n)=\overline\Sigma_{D_n}(\omega_n)\sqcup(\overline\Sigma_{D_n}(\omega_{n-1})+\overline\varepsilon_n).
$$
We have a $D_n$-module decomposition:
$$
V_{B_n}(\omega_n)= V_{D_n}(\omega_n)\oplus V_{D_n}(\omega_{n-1}),
$$
where the first summand is spanned by $v(\sigma), \overline\sigma\in \overline\Sigma_{D_n}(\omega_n)$, and the second summand is spanned by $v(\sigma), \overline\sigma\in \overline\Sigma_{D_n}(\omega_{n-1})+\overline\varepsilon_n$. The signatures $\sigma$ with $\overline\sigma\in\overline\Sigma_{D_n}(\omega_n)$ are essential for $D_n$ and have zero exponents corresponding to roots $\varepsilon_i$, hence they are essential for the extended monomial order. The signatures $\sigma$ with   $\overline\sigma\in\overline\Sigma_{D_n}(\omega_{n-1})+\overline\varepsilon_n$ are minimal among those for which $v(\sigma)$ span $V_{B_n}(\omega_n)$ modulo $V_{D_n}(\omega_n)$, hence they are essential, too.

The last equality we have to prove is:
$$
\overline\Sigma_{B_n}(2\omega_{n})=(\overline\Sigma_{D_n}(\omega_{n-1}+\omega_n)+\overline\varepsilon_n)\sqcup(\overline\Sigma_{D_n}(2\omega_{n-1})+2\overline\varepsilon_{n})\cup\overline\Sigma_{D_n}(2\omega_{n}).
$$
We have $$V_{B_n}(2\omega_{n})={\bigwedge}^{n}\mathbb{C}^{2n+1}={\bigwedge}^{n}\mathbb{C}^{2n}\oplus e_0\wedge{\bigwedge}^{n-1}\mathbb{C}^{2n}.$$ Moreover, ${\bigwedge}^{n}\mathbb{C}^{2n}=V_{D_n}(2\omega_{n-1})\oplus V_{D_n}(2\omega_n).$ Therefore
$$
V_{B_n}(2\omega_{n})=e_0\wedge V_{D_n}(\widehat\omega_{n-1})\oplus V_{D_n}(2\omega_{n-1})\oplus V_{D_n}(2\omega_{n}).
$$
The arguments similar to the previous cases finish the proof.
\end{proof}

By Lemma \ref{lemma} we have bijective maps (forgetting exponents corresponding to roots $\varepsilon_i$):
$$
\psi: \Sigma_{B_n}(\omega_p)\rightarrow\Sigma_{D_n}(\widehat\omega_p)\sqcup\Sigma_{D_n}(\omega_{p-1}), \quad p=1,\ldots,n,
$$
$$
\psi :
\Sigma_{B_n}(2\omega_{n})\rightarrow\Sigma_{D_n}(\omega_{n-1}+\omega_n)\sqcup\Sigma_{D_n}(2\omega_{n-1})\sqcup\Sigma_{D_n}(2\omega_{n}).
$$

Assume that we have two decompositions of some signature $\sigma$ of highest weight $\lambda=\sum k_i\omega_i=\sum l_i\varepsilon_i$ of $B_n$:
$$
\sigma=\sigma_1+\ldots+\sigma_k=\sigma'_1+\ldots+\sigma'_l,
$$
where $\sigma_i$ and $\sigma'_j$ are essential signatures for $B_n$ of highest weights in $\{\omega_1,\ldots,\omega_n,2\omega_n\}$.
 Then we can apply the map $\psi$ to these decompositions, and obtain two signatures of $D_n$:
$$
\psi(\sigma_1)+\ldots+\psi(\sigma_k)\quad\textrm{and}\quad\psi(\sigma'_1)+\ldots+\psi(\sigma'_l).
$$
We claim that these two signatures coincide. Obviously, these two signatures have the same exponents, hence we have to verify that the highest weights of these signatures coincide.
Let $s_i$ be the exponents of $\sigma$, corresponding to roots $\varepsilon_i$. It is easy to see that the highest weight of both signatures of $D_n$ is $\sum k_i'\omega_i=\sum l_i'\varepsilon_i$, where $l_i'=l_i-s_i$ and hence

\begin{equation}
\label{defw}
k'_i=k_i-s_i+s_{i+1},\quad i<n,\quad
k_{n}'=k_{n-1}-s_{n-1}+k_n-s_n.
\end{equation}
Hence the highest weights coincide and
$$
\psi(\sigma_1)+\ldots+\psi(\sigma_k)=\psi(\sigma'_1)+\ldots+\psi(\sigma'_l).
$$

Therefore we have a well-defined surjective homomorphism of semigroups:
$$
\psi:\Sigma_{B_n}^{(\cdot)}\rightarrow\Sigma_{D_n}^{(\cdot)}.
$$
The equations (\ref{defw}) imply that for any two $\sigma,\sigma'\in\Sigma^{(\cdot)}_{B_n}$, if $\psi(\sigma)=\psi(\sigma')$ and $\sigma,\sigma'$ have the same highest weight, then $\sigma=\sigma'$.

Notice, that by Lemma \ref{lemma} for any signature $\sigma\in\Sigma_{B_n}^{(\cdot)}$ the inequalities $k_i\geq s_i$ hold.

Let now $\tau\in\Sigma_{D_n}^{(\cdot)}$ be an arbitrary signature of highest weight $\sum k'_i\omega_i$, $s_i$ be some non-negative integers, and let $k_i$ be integers satisfying the equations (\ref{defw}).
\begin{lemma}
\label{lift}
Assume that $k_i\geq s_i$. Then there exists a unique $\sigma\in\Sigma_{B_n}^{(\cdot)}$ with highest weight $\sum k_i\omega_i$ such that $\psi(\sigma)=\tau$ and $s_i$ are the exponents of $\sigma$ corresponding to the roots $\varepsilon_i$.
\end{lemma}
We call $s_i$ the \emph{lifting parameters} for $\tau$.
\begin{proof}

Obviously, if $\sigma$ exists then it is unique.

The inequalities $k_i\geq s_i$, where $k_i,s_i\in\mathbb{R}_{\geq0}$, define a cone in $\mathbb{R}^{2n}$. The semigroup of integer points in this cone is generated by the following vectors given below by the values of their nonzero coordinates:
 \begin{enumerate}
 \item $k_i=1$ ($n$ vectors for $i=1,\ldots,n$),
 \item $k_i=s_i=1$ ($n$ vectors for $i=1,\ldots,n$).
 \end{enumerate}

 We can consider these vectors and hence any integer point in this cone as a tuple of exponents of an essential signature of $B_n$ of highest weight $\sum k_i\omega_i$ with exponents $s_i$ corresponding to short roots and zero exponents corresponding to long roots.

 Since the property $(\dagger)$ holds we may fix some decomposition of the signature $\tau=\tau_1+\ldots+\tau_m$, where $\tau_i$ are essential signatures of highest weights in $\{\omega_1,\ldots,\omega_{n-2},\widehat\omega_{n-1},\omega_{n-1},2\omega_{n-1}\}$ or in $\{\omega_1,\ldots,\omega_{n-2},\widehat\omega_{n-1},\omega_n,2\omega_n\}.$
  Let $\tilde\sigma=\tilde\sigma_1+\ldots+\tilde\sigma_{p}$ be a decomposition of a signature of highest weight $\sum k_i\omega_i$ with zero exponents, corresponding to long roots, and exponents $s_i$ corresponding to short roots via signatures $\tilde\sigma_i$ of fundamental highest weights corresponding to the generating vectors of the above semigroup. The signatures $\psi(\tilde\sigma)$ and $\tau$ have the same highest weight. Therefore for any signature $\tilde\sigma_i$ of highest weight $\omega_j, j< n-1$ one can find a signature $\tau_k$ such that $\psi(\tilde\sigma_i)$ and $\tau_k$ have the same highest weight. This fact is also true for a signature $\tilde\sigma_i$ of highest weight $\omega_{n-1}$ because of our choice of decomposition $\tau$. Combining (if needed) some pairs of signatures $\tilde\sigma_i$ of highest weight $\omega_n$ into signatures of highest weight $\widehat\omega_n$, one can obtain a decomposition with $p=m$ such that $\psi(\tilde\sigma_i)$ and $\tau_i$ have the same highest weight. By Lemma~\ref{lemma} there exists an essential signature $\sigma_i$ such that $\sigma_i$ and $\tilde\sigma_i$ have the same highest weight and exponents corresponding to short roots, and $\psi(\sigma_i)=\tau_i$. Now put $\sigma=\sigma_1+\ldots+\sigma_m$.

\end{proof}

\begin{sled}
\label{sledlift}
For any signature $\sigma\in\Sigma_{B_n}^{(\cdot)}$ and any decomposition of the signature $\psi(\sigma)=\tau_1+\ldots+\tau_m$, where all $\tau_i\in\Sigma'_{D_n}$ or all $\tau_i\in\Sigma''_{D_n}$, there exists a decomposition $\sigma=\sigma_1+\ldots+\sigma_m$ such that $\psi(\sigma_i)=\tau_i$ for all $i=1,\ldots,m$.

\end{sled}
\begin{proof}
 It follows directly from the proof of Lemma \ref{lift}.
\end{proof}
\begin{lemma}
\label{phi}
The property $(*)$ holds for $\Sigma_{B_n}^{(\cdot)}$.
\label{theorem2}
\end{lemma}
\begin{proof}

For any signature $\sigma$ of $B_{n}$ we denote by $s_{i}(\sigma)$ the exponents of $\sigma$ corresponding to the roots $\varepsilon_i$.
Assume that we have two decompositions of some signature $\sigma\in\Sigma_{B_n}^{(\cdot)}$ of highest weight $\lambda=\sum k_i\omega_i$:
$$
\sigma=\sigma_1+\ldots+\sigma_k=\sigma'_1+\ldots+\sigma'_l,
$$
where $\sigma_i$ and $\sigma'_j$ are essential signatures for $B_n$ of highest weights in $\{\omega_1,\ldots,\omega_n,2\omega_n\}$.

Apply the map $\psi$ to these decompositions:
$$
\psi(\sigma)=\psi(\sigma_1)+\ldots+\psi(\sigma_k)=\psi(\sigma'_1)+\ldots+\psi(\sigma'_l).
$$

The property ($\dagger$) implies that one can obtain other two decompositions of $\psi(\sigma)$ with all signatures in $\Sigma'_{D_n}$ or in $\Sigma''_{D_n}$ by applying admissible operations to the two initial decompositions. Indeed, one can replace a pair of signatures with highest weights $\{\omega_{n-1},\omega_n\}, \{2\omega_{n-1},2\omega_n\},\{2\omega_{n-1},\omega_n\},\{\omega_{n-1},2\omega_n\}$ with a signature of highest weight $\widehat\omega_{n-1}$ in the first case or with a pair of signatures with highest weights $\{\widehat\omega_{n-1},\widehat\omega_{n-1}\}, \{\widehat\omega_{n-1},\omega_{n-1}\},\{\widehat\omega_{n-1},\omega_{n}\}$, respectively. By Corollary \ref{sledlift} we can lift any such admissible operation to $B_n$. Thus we may assume that all signatures $\psi(\sigma_i),\psi(\sigma'_j)$ are in $\Sigma'_{D_n}$ or in $\Sigma''_{D_n}$.

 Since property ($*$) holds for $\Sigma'_{D_n}$ and $\Sigma''_{D_n}$ and each admissible operation on decompositions of signatures in $\Sigma'_{D_n}$ or in $\Sigma''_{D_n}$ is liftable to $\Sigma_{B_n}^{(\cdot)}$, we may assume that $k=l$ and
$$
\psi(\sigma_1)=\psi(\sigma'_1),\quad\ldots\quad,\quad\psi(\sigma_l)=\psi(\sigma'_l).
$$
We may suppose that there are no signatures of highest weight $\omega_n$ in these decompositions of $\sigma$. Indeed, assume that $\sigma_1$ has the highest weight $\omega_n$. Then, by Lemma \ref{lemma}, $\sigma'_1$ has the highest weight $\omega_n$, too, whence $\sigma_1=\sigma'_1$, and we may finish the proof by induction on $l$.
 Similarly, we may suppose that there are no signatures of highest weight $\widehat\omega_n$ with $s_n=0$ or $s_n=2$.

Assume that $\sigma_1$ has the highest weight $\widehat\omega_j$ such that $k_i=0, i<j$. Then either $\sigma'_1$ has the highest weight $\widehat\omega_j$, too, whence $\sigma'_1=\sigma_1$, and we are done by induction, or $\sigma'_1$ has the highest weight $\widehat\omega_{j+1}$ and $s_{j+1}(\sigma'_1)=1$, while $\sigma_1$ has zero exponents corresponding to the short roots. Therefore there exists a signature in the left-hand decomposition of $\sigma$, say $\sigma_2$, of highest weight $\widehat\omega_{j+1}$ such that $s_{j+1}(\sigma_2)=1$. We may swap the exponents of $\sigma_1$ and $\sigma_2$ corresponding to the long roots and obtain two new signatures $\widetilde\sigma_1$ and $\widetilde\sigma_2$. Obviously, $\widetilde\sigma_2=\sigma'_1$ and by Lemma \ref{lemma} $\widetilde\sigma_1$ is essential. Hence replacing $\sigma_1,\sigma_2$ with $\tilde\sigma_1,\tilde\sigma_2$ is an admissible operation. Induction on $l$ finishes the proof.
\end{proof}

\begin{lemma}
The property $(**)$ holds for $\Sigma^{(\cdot)}_{B_n}$.
\label{theorem3}
\end{lemma}
\begin{proof}

To prove property ($**$) we will compute the number of signatures in $\Sigma_{B_n}(\lambda_p)+\Sigma_{B_n}(\lambda_q)$, and compare it with dim$V_{B_n}(\lambda_p+\lambda_q)$ for all pairs $\lambda_p,\lambda_q$, except $\lambda_p=\lambda_q=\omega_n$.

 Consider, for example, the pair $\omega_p,\omega_q$. By Lemma \ref{lemma} we have  the following equalities:
$$
|\Sigma_{B_n}(\omega_p)+\Sigma_{B_n}(\omega_q)| = |\Sigma_{D_n}(\widehat\omega_p)+\Sigma_{D_n}(\widehat\omega_q)|+|\Sigma_{D_n}(\omega_{p-1})+\Sigma_{D_n}(\widehat\omega_q)|+ $$
$$
+|\Sigma_{D_n}(\widehat\omega_{p})+\Sigma_{D_n}(\omega_{q-1})|+|\Sigma_{D_n}(\omega_{p-1})+\Sigma_{D_n}(\omega_{q-1})|,\quad p\neq q;
$$
$$
|\Sigma_{B_n}(\omega_p)+\Sigma_{B_n}(\omega_p)| = |\Sigma_{D_n}(\widehat\omega_p)+\Sigma_{D_n}(\widehat\omega_p)|+|\Sigma_{D_n}(\omega_{p-1})+\Sigma_{D_n}(\widehat\omega_p)|+ $$
$$
+|\Sigma_{D_n}(\omega_{p-1})+\Sigma_{D_n}(\omega_{p-1})|.
$$

Since ($**$) holds for $\Sigma'_{D_n}$ and $\Sigma''_{D_n}$, it remains to verify the following equalities:
\begin{enumerate}
\item $\dim V_{B_n}(\omega_p+\omega_q)=\dim V_{D_n}(\widehat\omega_p+\widehat\omega_q)+\dim V_{D_n}(\widehat\omega_p+\omega_{q-1})+
    \newline
    +\dim V_{D_n}(\omega_{p-1}+\widehat\omega_q)+\dim V_{D_n}(\omega_{p-1}+\omega_{q-1}),\quad p\neq q;$
 \item  $\dim V_{B_n}(2\omega_p)=\dim V_{D_n}(2\widehat\omega_p)+\dim V_{D_n}(\widehat\omega_p+\omega_{p-1})+\dim V_{D_n}(2\omega_{p-1})$.
\end{enumerate}
One can easily verify that all equalities above are identities.

We use similar arguments for pairs $2\omega_n,\omega_p$, $p\neq n$.
For the pairs $2\omega_n$, $2\omega_n$ and $2\omega_n,\omega_n$ one should use the same arguments with the property ($\dagger$), which implies that $$\Sigma_{D_n}(2\omega_{n-1})+\Sigma_{D_n}(2\omega_n)\subset\Sigma_{D_n}(2\widehat\omega_{n-1})=\Sigma_{D_n}(\widehat\omega_{n-1})+\Sigma_{D_n}(\widehat\omega_{n-1}),$$
$$\Sigma_{D_n}(2\omega_{n-1})+\Sigma_{D_n}(\omega_n)\subset\Sigma_{D_n}(\widehat\omega_{n-1}+\omega_{n-1})=\Sigma_{D_n}(\widehat\omega_{n-1})+\Sigma_{D_n}(\omega_{n-1}),$$
$$\Sigma_{D_n}(2\omega_{n})+\Sigma_{D_n}(\omega_{n-1})\subset\Sigma_{D_n}(\widehat\omega_{n-1}+\omega_{n})=\Sigma_{D_n}(\widehat\omega_{n-1})+\Sigma_{D_n}(\omega_{n}).$$

By \cite[Ref. Chap.,\S 2, Table 5]{[ÂÎ]} we have the formulas for all dim$V_{B_{n}}(\lambda_i+\lambda_j)$ and dim$V_{D_{n}}(\lambda_i+\lambda_j)$.
 One can easily verify that similar equalities for the remaining pairs of
$\lambda_i,\lambda_j$ are true, too. Therefore ($**$) holds.
Notice that these equalities also follow from branching rules for orthogonal Lie algebras \cite[Chap. XVIII, \S 129]{[Z]}.
\end{proof}
\begin{sled}
\label{conj1bn}
The semigroup of essential signatures for $B_n$ is generated by essential signatures of highest weights in $\{\omega_1,\ldots,\omega_n,2\omega_n\}$.

\end{sled}
 Denote by ($\sharp'$) the inequalities on the coordinates of a signature which define the semigroup $\Sigma_{D_n}$. Let $\sum k'_i\omega_i$ denote the highest weight of a signature for $D_n$. Denote by ($\natural$) the inequalities obtaining from ($\sharp'$) by expressing $k'_i$ via $k_i,s_i$ using the formulas (\ref{defw}).
\begin{proposition}
\label{conj3bn}
The
semigroup $\Sigma_{B_n}$ is given by the inequalities $(\natural)$ and $k_i\geq s_i$, $i=1,\ldots,n$.
\end{proposition}
\begin{proof}
Let $\sigma\in\Sigma_{B_n}$ be a signature of highest weight $\sum k_i\omega_i$. Then by Corollary \ref{conj1bn} one has $\sigma\in\Sigma_{B_n}^{(\cdot)}$. Therefore inequalities $k_i\geq s_i$ hold, where $s_i$ are the exponents corresponding to the short roots. The inequalities ($\natural$) hold, because inequalities ($\sharp'$) hold for $\psi(\sigma)$.

Conversely, suppose that the inequalities ($\natural$) and $k_i\geq s_i$ hold for some signature $\sigma$ of highest weight $\sum k_i\omega_i$. Let $\tau$ be the signature of $D_n$ of highest weight $\sum k_i'\omega_i$, where $k_i'$ are defined by the equalities (\ref{defw}), such that $\tau$ and $\sigma$ have the same exponents corresponding to long roots. The signature $\tau$ is essential since the inequalities ($\sharp'$) hold. Then since inequalities $k_i\geq s_i$ hold  one can lift the signature $\tau$ to an essential signature $\sigma'$ of highest weight $\sum k_i\omega_i$ with lifting parameters $s_i$ by Lemma \ref{lift}. Obviously, $\sigma=\sigma'$. Therefore  $\sigma\in\Sigma_{B_n}$.
\end{proof}
\begin{sled}
The semigroup $\Sigma_{B_n}$ is saturated.
\label{theorem4}
\end{sled}

\subsection {From $D_{n}$ to $D_{n+1}$}
In this subsection we make an induction step in our argument. We assume that we have some numeration of positive roots and some monomial order for $D_{n}$ such that the inductive hypothesis for $D_n$ holds. Here we prove the inductive hypothesis for $D_{n+1}$.

 We have the standard embedding of $D_{n}$ in $D_{n+1}$ such that the following $D_{n}$-module decomposition holds:
$$
V_{D_{n+1}}(\omega_1)=V_{D_{n}}(\omega_1)\oplus \langle e_{n+1}\rangle\oplus\langle e_{-(n+1)}\rangle.
$$
 Since the Cartan subalgebra of $D_{n+1}$ contains the Cartan subalgebra of $D_{n}$ and normalizes $D_{n}$, we can consider the roots of $D_{n}$ as roots of $D_{n+1}$.
 First of all we need to extend the monomial order and the numeration of positive roots from $D_{n}$ to $D_{n+1}$. We extend the numeration as follows:
$$
\textrm{positive roots of $D_{n}$},\varepsilon_1-\varepsilon_{n+1},\ldots,\varepsilon_{n}-\varepsilon_{n+1},\varepsilon_1+\varepsilon_{n+1},\ldots,\varepsilon_{n}+\varepsilon_{n+1}.
$$
We extend the monomial order (on signatures with the same highest weight) as follows:
$$
(\overline\sigma,k_1,\ldots,k_{n+1},l_1,\ldots,l_{n+1})< (\overline\sigma',k_1',\ldots,k_{n+1}',l_1',\ldots,l_{n+1}')
$$
if either $(l_1,\ldots,l_{n+1})<(l_1',\ldots,l_{n+1}')$ in degree lexicographic order, or $(l_1,\ldots,l_{n+1})=(l_1',\ldots,l_{n+1}')$ and $(k_1,\ldots,k_{n+1})<(k_1',\ldots,k_{n+1}')$ in degree lexicographic order, or $(k_1,\ldots,k_{n+1},l_1,\ldots,l_{n+1})=(k_1',\ldots,k_{n+1}',l_1',\ldots,l_{n+1}')$ and $\overline\sigma<\overline\sigma'$ with respect to the monomial order for $D_{n}$.

Now we describe the set of essential signatures of highest weights
$\lambda_i$ for $D_{n+1}$. For any signature $\sigma$ of $D_{n}$ we consider $\overline\sigma$ as the tuple $(\overline\sigma,0,\ldots,0)$ for $D_{n+1}$.
\begin{lemma}
\label{lemma2}
\begin{multline*}
$$
\overline\Sigma_{D_{n+1}}(\omega_p)=\overline\Sigma_{D_{n}}(\widehat\omega_p)\cup(\overline\Sigma_{D_{n}}(\omega_{p-1})+\overline{\varepsilon_p-\varepsilon_{n+1}})
\cup(\overline\Sigma_{D_{n}}(\omega_{p-1})+\overline{\varepsilon_p+\varepsilon_{n+1}})\cup\\
\cup(\overline\Sigma_{D_{n}}(\omega_{p-2})+\overline{\varepsilon_{p-1}-\varepsilon_{n+1}}+\overline{\varepsilon_{p}+\varepsilon_{n+1}}),\quad p<n;
$$
\end{multline*}
$$
\overline\Sigma_{D_{n+1}}(\omega_{n})=\overline\Sigma_{D_{n}}(\omega_{n})\cup(\overline\Sigma_{D_{n}}(\omega_{n-1})+\overline{\varepsilon_{n}-\varepsilon_{n+1}}),
$$
$$
\overline\Sigma_{D_{n+1}}(\omega_{n+1})=\overline\Sigma_{D_{n}}(\omega_{n})\cup(\overline\Sigma_{D_{n}}(\omega_{n-1})+\overline{\varepsilon_{n}+\varepsilon_{n+1}}),
$$
\begin{multline*}
$$
\overline\Sigma_{D_{n+1}}(2\omega_{n})=\overline\Sigma_{D_{n}}(2\omega_{n})\cup(\overline\Sigma_{D_{n}}(\widehat\omega_{n-1})+\overline{\varepsilon_{n}-\varepsilon_{n+1}})
\cup\\
\cup(\overline\Sigma_{D_{n}}(2\omega_{n-1})+2(\overline{\varepsilon_{n}-\varepsilon_{n+1}})),
$$
\end{multline*}
\begin{multline*}
$$
\overline\Sigma_{D_{n+1}}(2\omega_{n+1})=\overline\Sigma_{D_{n}}(2\omega_{n})\cup(\overline\Sigma_{D_{n}}(\widehat\omega_{n-1})+\overline{\varepsilon_{n}+\varepsilon_{n+1}})
\cup\\
\cup(\overline\Sigma_{D_{n}}(2\omega_{n-1})+2(\overline{\varepsilon_{n}+\varepsilon_{n+1}})),
$$
\end{multline*}
\begin{multline*}
$$
\overline\Sigma_{D_{n+1}}(\widehat\omega_{n})=\overline\Sigma_{D_{n}}(2\omega_{n})\cup(\overline\Sigma_{D_{n}}(2\omega_{n-1})+\overline{\varepsilon_{n}-\varepsilon_{n+1}}+\overline{\varepsilon_{n}+\varepsilon_{n+1}})\cup\\
\cup(\overline\Sigma_{D_{n}}(\widehat\omega_{n-1})+\overline{\varepsilon_{n}-\varepsilon_{n+1}})\cup(\overline\Sigma_{D_{n}}(\widehat\omega_{n-1})+\overline{\varepsilon_{n}+\varepsilon_{n+1}})\cup\\
\cup(\overline\Sigma_{D_{n}}(\omega_{n-2})+\overline{\varepsilon_{n-1}-\varepsilon_{n+1}}+\overline{\varepsilon_{n}+\varepsilon_{n+1}}),
$$
\end{multline*}
\end{lemma}
$$\textrm{where}\quad\overline\Sigma_{D_{n}}(\omega_0)=\overline\Sigma_{D_{n}}(\omega_{-1}):=\{(0,\ldots,0)\},$$
$$\overline{\varepsilon_0-\varepsilon_{n+1}}:=(0,\ldots,0).$$
\begin{proof}
For any $p<n$ one has a $D_{n}$-module decomposition:
\begin{multline*}
$$
V_{D_{n+1}}(\omega_p)=V_{D_{n}}(\widehat\omega_p)\oplus
   e_{n+1}\wedge V_{D_{n}}(\omega_{p-1})\oplus\\  \oplus e_{-(n+1)}\wedge V_{D_{n}}(\omega_{p-1})\oplus e_{n+1}\wedge e_{-(n+1)}\wedge V_{D_{n}}(\omega_{p-2}),
$$
\end{multline*}
where the first summand is spanned by the vectors $v(\sigma), \overline\sigma\in\overline\Sigma_{D_{n}}(\omega_p)$, the second and third summands are spanned by $v(\sigma),\overline\sigma\in\overline\Sigma_{D_{n}}(\omega_{p-1})+\overline{\varepsilon_p\mp\varepsilon_{n+1}}$, respectively, and the last summand is spanned modulo the previous summands by the vectors $v(\sigma),\overline\sigma\in\overline\Sigma_{D_{n}}(\omega_{p-2})+\overline{\varepsilon_{p-1}-\varepsilon_{n+1}}+\overline{\varepsilon_{p}+\varepsilon_{n+1}}$. One can prove that all these signatures $\sigma$ are essential with respect to the extended monomial order by using the arguments similar to those in Lemma~\ref{lemma}. Therefore the first equality holds.

 The next two equalities are similar to the following equality in Lemma~\ref{lemma}:
 $$
 \overline\Sigma_{B_{n}}(\omega_{n})=\overline\Sigma_{D_{n}}(\omega_{n})\cup(\overline\Sigma_{D_{n}}(\omega_{n-1})+\overline{\varepsilon}_{n}),
 $$
 We prove them by using the same arguments as in Lemma \ref{lemma} replacing $\varepsilon_{n}$ by $\varepsilon_{n}\pm \varepsilon_{n+1}$.

 Now we prove the equality
 \begin{multline*}
$$
\overline\Sigma_{D_{n+1}}(2\omega_{n+1})=\overline\Sigma_{D_{n}}(2\omega_{n})\cup(\overline\Sigma_{D_{n}}(\widehat\omega_{n-1})+\overline{\varepsilon_{n}+\varepsilon_{n+1}})
\cup\\
\cup(\overline\Sigma_{D_{n}}(2\omega_{n-1})+2(\overline{\varepsilon_{n}+\varepsilon_{n+1}})).
$$
\end{multline*}
 The highest weight vector of $V_{D_n}(2\omega_{n})$ is $v_{2\omega_n}=e_1\wedge\ldots\wedge e_n$ and the highest weight vector of $V_{D_{n+1}}(2\omega_{n+1})$ is $v_{2\omega_n}\wedge e_{n+1}$. The signatures $\sigma$ with $\overline\sigma\in\overline\Sigma_{D_{n}}(2\omega_{n})$ have zero exponents corresponding to roots $\varepsilon_i\pm\varepsilon_{n+1}$, hence they are minimal among signatures for which $v(\sigma)$ span $V_{D_n}(2\omega_{n})\wedge e_{n+1}$.  Therefore they are essential for the extended monomial order.

 The signatures $\sigma$ with exponent 1 corresponding to the root $\varepsilon_{n}+\varepsilon_{n+1}$ are minimal signatures of the weights $\sum_{i=1}^{n}k_i\varepsilon_i+0\cdot\varepsilon_{n+1}$,  $k_i\in\{-1,0,1\}$. After applying the lowering operator corresponding to the root $\varepsilon_{n}+\varepsilon_{n+1}$ to the highest weight vector $e_1\wedge\ldots\wedge e_{n+1}$ one obtains the vector $e_1\wedge\ldots\wedge e_{n-1}\wedge e_{-n}\wedge e_{n}+e_1\wedge\ldots\wedge e_{n-1}\wedge e_{-(n+1)}\wedge e_{n+1}$. The second summand has the form $v_{\widehat\omega_{n-1}}\wedge e_{-(n+1)}\wedge e_{n+1}$, where $v_{\widehat\omega_{n-1}}$ is the highest weight vector for $V_{D_n}(\widehat\omega_{n-1})$. The vectors $v(\sigma)$ with $\overline\sigma\in\overline\Sigma_{D_{n}}(\widehat\omega_{n-1})+\overline{\varepsilon_{n}+\varepsilon_{n+1}}$ span $V_{D_n}(\widehat\omega_{n-1})\wedge e_{-(n+1)}\wedge e_{n+1}$ modulo $\bigwedge^{n}\mathbb{C}^{2n}$, where $\mathbb{C}^{2n}=\langle e_{\pm1},\ldots,e_{\pm n} \rangle$. Hence the corresponding signatures are essential.

 Similarly, the signatures $\sigma$ with exponent 2 corresponding to the root $\varepsilon_{n}+\varepsilon_{n+1}$ are minimal signatures of the weights $\sum_{i=1}^{n}k_i\varepsilon_i-\varepsilon_{n+1}$, $k_i\in\{-1,0,1\}$. After applying two lowering operators corresponding to the root $\varepsilon_{n}+\varepsilon_{n+1}$ to the highest weight vector $e_1\wedge\ldots\wedge e_{n+1}$ one obtains the vector $e_1\wedge\ldots\wedge e_{n-1}\wedge e_{-n}\wedge e_{-(n+1)}=v_{2\omega_{n-1}}\wedge e_{-(n+1)}$, where $v_{2\omega_{n-1}}$ is the highest weight vector of $V_{D_n}(2\omega_{n-1})$. The vectors $v(\sigma)$ with $\overline\sigma\in\overline\Sigma_{D_{n}}(2\omega_{n-1})+2(\overline{\varepsilon_{n}+\varepsilon_{n+1}})$ span $V_{D_n}(2\omega_{n-1})\wedge e_{-(n+1)}$. Hence the corresponding signatures
 are essential.

 Thus all signatures on the right-hand side are essential. By the second equality from Lemma \ref{lemma} the number of these signatures equals $\dim V_{B_n}(2\omega_n)$.
 Therefore, since $\dim V_{B_n}(2\omega_n)=\dim V_{D_{n+1}}(2\omega_{n+1})$ we have found all essential signatures of highest weight $2\omega_{n+1}$. The proof for $\overline\Sigma_{D_{n+1}}(2\omega_n)$ is similar.

The proof of the last equality is similar to the one for the first equality. One has the  $D_n$-module decomposition
 \begin{multline*}
 $$
 V_{D_{n+1}}(\widehat\omega_{n})=V_{D_{n}}(2\omega_{n})\oplus V_{D_{n}}(2\omega_{n-1}) \oplus e_{n+1}\wedge V_{D_{n}}(\widehat\omega_{n-1})\oplus \\\oplus e_{-(n+1)}\wedge V_{D_{n}}(\widehat\omega_{n-1})\oplus e_{n+1}\wedge e_{-(n+1)}\wedge V_{D_{n}}(\omega_{n-2})
 $$
 \end{multline*}
 similar to one for the first equality, where the first summand in the right-hand side  $V_{D_{n}}(\widehat\omega_{p})={\bigwedge}^p\mathbb{C}^{2n}$ is replaced by $V_{D_{n}}(2\omega_{n})\oplus V_{D_{n}}(2\omega_{n-1})={\bigwedge}^{n}\mathbb{C}^{2n}$.
\end{proof}

\begin{lemma}
The property $(\dagger)$ holds for $D_{n+1}$.
\end{lemma}
\begin{proof}
It is enough to show that any sum of two essential signatures of highest weights in $\{\omega_n,\omega_{n+1},2\omega_{n},2\omega_{n+1},\widehat\omega_n\}$ belongs to $\Sigma'_{D_{n+1}}$ or to $\Sigma''_{D_{n+1}}$.

Let, for example, $\sigma\in\Sigma_{D_{n+1}}(2\omega_n)$ and $\sigma'\in\Sigma_{D_{n+1}}(2\omega_{n+1})$. Then by the previous Lemma one has:
$$\overline\sigma\in\overline\Sigma_{D_n}(2\omega_n),\;\textrm{or}\; \overline\sigma\in\overline\Sigma_{D_n}(\widehat\omega_{n-1})+\overline{\varepsilon_n-\varepsilon_{n+1}},
\
\; \textrm{or}\;  \overline\sigma\in\overline\Sigma_{D_n}(2\omega_{n-1})+2(\overline{\varepsilon_n-\varepsilon_{n+1}})$$ and $$\overline\sigma'\in\overline\Sigma_{D_n}(2\omega_n),\;\textrm{or}\; \overline\sigma'\in\overline\Sigma_{D_n}(\widehat\omega_{n-1})+\overline{\varepsilon_n+\varepsilon_{n+1}},\;\textrm{or}\; \overline\sigma'\in\overline\Sigma_{D_n}(2\omega_{n-1})+2(\overline{\varepsilon_n+\varepsilon_{n+1}}).$$
The property ($\dagger$) for $D_n$ and description of $\overline\Sigma_{D_{n+1}}(\widehat\omega_n)$ imply that in any case $\sigma+\sigma'\in\Sigma_{D_{n+1}}(\widehat\omega_n)+\Sigma_{D_{n+1}}(\widehat\omega_n)$. Hence $\sigma+\sigma'$ belongs to both $\Sigma'_{D_{n+1}}$ and $\Sigma''_{D_{n+1}}$.

One can easily verify that the other cases are similar to the one in this example.

\end{proof}
By Lemma \ref{lemma2} we have the surjective maps (forgetting exponents corresponding to roots $\varepsilon_i\pm\varepsilon_{n+1}$):

$$
\psi:\Sigma_{D_{n+1}}(\omega_p)\rightarrow\Sigma_{D_{n}}(\widehat\omega_p)
\sqcup\Sigma_{D_{n}}(\omega_{p-1})\sqcup\Sigma_{D_{n}}(\omega_{p-2}),\quad p<n,
$$
$$
\psi:\Sigma_{D_{n+1}}(\omega_{n})\rightarrow\Sigma_{D_{n}}(\omega_{n})\sqcup\Sigma_{D_{n}}(\omega_{n-1}),
$$
$$
\psi:\Sigma_{D_{n+1}}(\omega_{n+1})\rightarrow\Sigma_{D_{n}}(\omega_{n})\sqcup\Sigma_{D_{n}}(\omega_{n-1}),
$$
$$
\psi:\Sigma_{D_{n+1}}(2\omega_{n})\rightarrow\Sigma_{D_{n}}(\widehat\omega_{n-1})
\sqcup\Sigma_{D_{n}}(2\omega_{n-1})\sqcup\Sigma_{D_{n}}(2\omega_{n}),
$$
$$
\psi:\Sigma_{D_{n+1}}(2\omega_{n+1})\rightarrow\Sigma_{D_{n}}(\widehat\omega_{n-1})\sqcup
\Sigma_{D_{n}}(2\omega_{n-1})\sqcup\Sigma_{D_{n}}(2\omega_{n}),
$$
$$
\psi:\Sigma_{D_{n+1}}(\widehat\omega_{n})\rightarrow\Sigma_{D_{n}}(2\omega_{n})\sqcup\Sigma_{D_{n}}(2\omega_{n-1})
\sqcup\Sigma_{D_{n}}(\widehat\omega_{n-1})
\sqcup\Sigma_{D_{n}}(\omega_{n-2}).
$$

As in 4.2 we can extend these maps to a homomorphism of semigroups $\Sigma_{D_{n+1}}^{(\cdot)}\rightarrow\Sigma_{D_{n}}^{(\cdot)}$. Indeed, assume that we have two decompositions of some signature $\sigma$ of highest weight $\lambda=\sum k_i\omega_i$ in $\Sigma_{D_{n+1}}^{(\cdot)}$:
$$
\sigma_1+\ldots+\sigma_k=\sigma'_1+\ldots+\sigma'_l,
$$
where $\sigma_i$ and $\sigma'_j$ are essential signatures for $D_{n+1}$ of highest weights in $$\{\omega_1,\ldots,\omega_{n+1},2\omega_{n},2\omega_{n+1},\widehat\omega_{n}\}.$$
 Then we can apply the map $\psi$ to these decompositions, and obtain two signatures in $\Sigma_{D_{n}}^{(\cdot)}$:
$$
\psi(\sigma_1)+\ldots+\psi(\sigma_k)\quad
\textrm{and}\quad\psi(\sigma'_1)+\ldots+\psi(\sigma'_l).
$$
We claim that these two signatures coincide. Obviously, these two signatures have the same exponents, hence we have to verify that the highest weights of these signatures coincide.
Let $s^{\pm}_i$ be the exponents of $\sigma$, corresponding to roots $\varepsilon_i\pm\varepsilon_{n+1}$ respectively.
Notice, that if $\sum^{n+1}_{i=1} k_i\omega_i=\sum^{n+1}_{i=1} l_i\varepsilon_i$ then the highest weight of both signatures in $\Sigma_{D_{n}}^{(\cdot)}$ is $\sum^{n}_{i=1} k'_i\omega_i=\sum^{n}_{i=1} l'_i\varepsilon_i$, where
$$
l'_i=l_i-s^{-}_i-s^{+}_i,\quad i\leq n.
$$
It is easy to deduce that
\begin{equation}
\begin{split}
\label{defw2}
k'_i=k_i-s^{-}_i-s^{+}_i+s^{-}_{i+1}+s^{+}_{i+1},\quad i<n,\\
k_{n}'=k_{n-1}+k_{n}+k_{n+1}-s^{-}_{n}-s^{+}_{n}-s^{-}_{n-1}-s^{+}_{n-1}.
\end{split}
\end{equation}

Hence
$$
\psi(\sigma_1)+\ldots+\psi(\sigma_k)=\psi(\sigma'_1)+\ldots+\psi(\sigma'_l).
$$

Therefore we have well-defined surjective homomorphism of semigroups:
$$
\psi:\Sigma_{D_{n+1}}^{(\cdot)}\rightarrow\Sigma_{D_{n}}^{(\cdot)}.
$$

Notice, that by Lemma \ref{lemma2} for $\{k_i\},\{k'_i\},\{s^{-}_i\},\{s^{+}_i\}$ the following inequalities hold:
\begin{equation}
\label{ineq}
\begin{split}
s^{+}_i&\leq k_i,\quad i<n,\\
s^{-}_i+s^{+}_{i}-s^{+}_{i+1}&\leq k_i,\quad i<n\\
s^{+}_{n-1}+s^{-}_{n-1}+s^{-}_{n}&\leq k_{n-1}+k_{n},\\
s^{-}_{n}&\leq k_{n},\\
s^{+}_{n}&\leq k_{n+1}.
\end{split}
\end{equation}
These inequalities together with $s^{\pm}_i,k_i\geq0$ define a cone $C\subset\mathbb{R}^{3n+1}$. Consider the following lattice vectors in this cone given below by the values of their nonzero coordinates:

\begin{equation}
\label{vectors}
\begin{split}
 1.&\quad k_i=1\quad(n+1\quad\textrm{vectors for}\quad i=1,\ldots,n+1),\\
 2.&\quad k_i=s^{+}_i=1\quad(n-1\quad\textrm{vectors for}\quad i=1,\ldots,n-1),\\
 3.&\quad k_i=s^{-}_i=1\quad(n\quad\textrm{vectors for}\quad i=1,\ldots,n),\\
 4.&\quad k_i=s^{+}_i=s^{-}_{i-1}=1\quad(n-2\quad\textrm{vectors for}\quad i=2,\ldots,n-1),\\
 5.&\quad k_{n+1}=s^{+}_{n}=1,\\
 6.&\quad k_{n+1}=k_{n}=s^{+}_{n}=s^{-}_{n-1}=1.
 \end{split}
\end{equation}
\begin{lemma}
The vectors $(\ref{vectors})$ generate the semigroup of integer points in the cone $C$.
\end{lemma}
\begin{proof}
We start with an arbitrary integer point in the cone. So we have tuples of non-negative integers $\{k_i\},\{s^{\pm}_{i}\}$, satisfying the inequalities (\ref{ineq}). We show how to subtract one of the vectors (\ref{vectors}) to get an integer point in the cone again. We provide a step by step algorithm below. In every step one can easily check that inequalities (\ref{ineq}) are preserved after subtraction.

\begin{enumerate}
\item Subtract the vector 3 with $i=n$ until $s^{-}_{n}=0$.
\item Subtract the vector 5 until
\item[] $s^{+}_{n}=0$
\item[] or
\item[] $s^{-}_{n-1}+s^{+}_{n-1}-s^{+}_{n}=k_{n-1}$. In the latter case, the fourth inequality is follows from the third one. Moreover, one has $s^{+}_{n}\leq s^{-}_{n-1}$, because otherwise $s^{+}_{n-1}=k_{n-1}+s^{+}_{n}-s^{-}_{n-1}> k_{n-1}$. A contradiction with the first inequality with $i=n-1$. Then subtract the vector 6 until $s^{+}_{n}=0$.
\item Subtract the vector 1 with $i=n,n+1$ until $k_{n}=k_{n+1}=0$.

\item Subtract the vector 3 with $i=n-1$ until $s^{-}_{n-1}=0$.
\item Subtract the vector 2 with $i=n-1$ until
\item[] $s^{+}_{n-1}=0$
\item[] or
\item[] $s^{-}_{n-2}+s^{+}_{n-2}-s^{+}_{n-1}=k_{n-2}$. In the latter case one has $s^{+}_{n-1}\leq s^{-}_{n-2}$, because otherwise $s^{+}_{n-2}=k_{n-2}+s^{+}_{n-1}-s^{-}_{n-2}> k_{n-2}$. A contradiction with the first inequality with $i=n-2$. Then subtract the vector 4 with $i=n-1$ until $s^{+}_{n-1}=0$.
\item Subtract the vector 1 with $i=n-1$ until $k_{n-1}=0$.
\item[]\ldots\ldots
\end{enumerate}
We repeat 4-6 decreasing $i$ by 1 in every step.
\end{proof}
Let now $\tau\in\Sigma_{D_{n}}^{(\cdot)}$ be a signature of highest weight $\sum k'_i\omega_i$, $s^{\pm}_i$ be some non-negative integers. Let $k_i$ be integers satisfying the equations (\ref{defw2}).
\begin{lemma}
\label{lift2}
Assume that inequalities $(\ref{ineq})$ hold. Then there exists a unique $\sigma\in\Sigma_{D_{n+1}}^{(\cdot)}$ of highest weight $\sum k_i\omega_i$ such that $\psi(\sigma)=\tau$, and $s^{\pm}_i$  are the exponents of $\sigma$ corresponding to the roots $\varepsilon_i\pm\varepsilon_{n+1}$ respectively.
\end{lemma}
We call $s_i^{\pm}$ the \emph{lifting parameters} for $\tau$.
\begin{proof}
Obviously, if $\sigma$ exists then it is unique.

We argue similarly to the proof of Lemma \ref{lift}. The inequalities (\ref{ineq}), where $s^{\pm}_i, k_i\in\mathbb{R}_{\geq0}$ define the cone $C$. The semigroup of integer points in $C$ is generated by the vectors (\ref{vectors}). We can consider these vectors and hence any integer point in $C$ as an essential signature of $D_{n+1}$ of highest weight $\sum k_i\omega_i$ with exponents $s^{\pm}_i$ corresponding to roots $\varepsilon_i\pm\varepsilon_{n+1}$, and zero exponents corresponding to other roots.

 Fix some decomposition of the signature $\tau=\tau_1+\ldots+\tau_m$, where $\tau_i$ are essential signatures of highest weights in $\{\omega_1,\ldots,\omega_{n-1},2\omega_{n-1},\widehat\omega_{n-1}\}$ or in $\{\omega_1,\ldots,\omega_{n-2},\omega_n,2\omega_{n},\widehat\omega_{n-1}\}$.
  Let $\tilde\sigma=\tilde\sigma_1+\ldots+\tilde\sigma_{p}$ be a decomposition of a signature of highest weight $\sum k_i\omega_i$ with exponents $s^{\pm}_i$ corresponding to roots $\varepsilon_i\pm\varepsilon_{n+1}$ and zero exponents corresponding to other roots via signatures $\tilde\sigma_i$  corresponding to the generating vectors of the above semigroup. The signatures $\psi(\tilde\sigma)$ and $\tau$ have the same highest weight. Therefore for any signature $\tilde\sigma_i$ of highest weight $\omega_j, j<n$ one can find a signature $\tau_k$ such that $\psi(\tilde\sigma_i)$ and $\tau_k$ have the same highest weight. This fact is also true for a signature $\tilde\sigma_i$ of highest weight $\widehat\omega_{n}$ because of our choice of decomposition $\tau$. Combining (if needed) some pairs of signatures $\tilde\sigma_i$ of highest weights $\omega_{n},\omega_{n+1}$ into signatures of highest weights $\widehat\omega_{n},2\omega_{n},2\omega_{n+1}$, one can obtain a decomposition with $p=m$ such that $\psi(\tilde\sigma_i)$ and $\tau_i$ have the same highest weight. By Lemma \ref{lemma2} there exists an essential signature $\sigma_i$ such that $\sigma_i$ and $\tilde\sigma_i$ have the same highest weight and exponents corresponding to the roots $\varepsilon_i\pm\varepsilon_{n+1}$, and $\psi(\sigma_i)=\tau_i$. Now put $\sigma=\sigma_1+\ldots+\sigma_m$.

\end{proof}
\begin{sled}
\label{sledlift2}
For any signature $\sigma\in\Sigma_{D_{n+1}}^{(\cdot)}$ and any decomposition of the signature $\psi(\sigma)=\tau_1+\ldots+\tau_m$, where all signatures $\tau_i$ have highest weights in $\{\omega_1,\ldots,\omega_{n-1},2\omega_{n-1},\widehat\omega_{n-1}\}$ (resp. in $\{\omega_1,\ldots,\omega_{n-2},\omega_n,2\omega_{n},\widehat\omega_{n-1}\}$) and at most one signature has the highest weight $\omega_{n-1}$ (resp. $\omega_n$), there exists a decomposition $\sigma=\sigma_1+\ldots+\sigma_m$ such that all $\sigma_i$ have highest weights in $\{\omega_1,\ldots,\omega_{n},2\omega_{n},\widehat\omega_{n}\}$ or in $\{\omega_1,\ldots,\omega_{n-1},\omega_{n+1},2\omega_{n+1},\widehat\omega_{n}\}$ and $\psi(\sigma_i)=\tau_i$ for all $i=1,\ldots,m$.

\end{sled}
\begin{proof}
Let $\sigma\in\Sigma'_{D_{n+1}}$ and all $\tau_i$ have highest weights in $\{\omega_1,\ldots,\omega_{n-1},2\omega_{n-1},\widehat\omega_{n-1}\}$. We consider the signature $\tilde\sigma=\tilde\sigma_1+\ldots+\tilde\sigma_{p}$ from the proof of previous Lemma, where $\sum k_i\omega_i$ is the highest weight of $\sigma$ and $s^{\pm}_i$ are exponents of $\sigma$ corresponding to the roots $\varepsilon_i\pm\varepsilon_{n+1}$.

The proof of previous Lemma imply that for any signature $\tilde\sigma_i$ of highest weight $\widehat\omega_j, j\leq n$ one can find a signature $\tau_k$ such that $\psi(\tilde\sigma_i)$ and $\tau_k$ have the same highest weight. We want to combine all signatures (except at most one) $\tilde\sigma_i,\tilde\sigma_j$ of highest weights $\omega_{n+1},\omega_n$ or $\omega_n,\omega_n$ into signatures, say $\tilde\sigma'_i$, of highest weights $\widehat\omega_n$ or $2\omega_n$ such that $\psi(\tilde\sigma'_i)$ have the same highest weight as some $\tau_k$. Then we may finish the proof as in previous Lemma. Thus we may assume that all signatures $\tilde\sigma_i$ have the highest weights $\omega_{n},\omega_{n+1}$.

The number of signatures $\tilde\sigma_i$ of highest weight $\omega_{n}$ is greater than the number of signatures with highest weight $\omega_{n+1}$ since $\sigma\in\Sigma'_{D_{n+1}}$. Therefore one can combine these signatures into pairs $\{\tilde\sigma_i,\tilde\sigma_j\}$ such that in every pair the highest weights are $\omega_{n},\omega_{n+1}$ or $\omega_{n},\omega_n$ and we have at most one signature of highest weight $\omega_{n}$ without a pair. Moreover, we may assume that in every such pair the highest weights of $\psi(\tilde\sigma_i)$, $\psi(\tilde\sigma_j)$ are $\omega_{n-1},\omega_{n}$ or $\omega_{n-1},\omega_{n-1}$ since $\tau_i\in\Sigma'_{D_{n}}$ and $\psi(\tilde\sigma)$ and $\psi(\sigma)$ have the same highest weight. Indeed, for every pair $\{\tilde\sigma_i,\tilde\sigma_j\}$ such that the highest weight of both $\psi(\tilde\sigma_i)$, $\psi(\tilde\sigma_j)$ is $\omega_n$ one has a pair $\{\tilde\sigma'_i,\tilde\sigma'_j\}$ such that the highest weight of both $\psi(\tilde\sigma'_i)$, $\psi(\tilde\sigma'_j)$ is $\omega_{n-1}$. Then we may swap any signatures of highest weights $\omega_n$ from the first pair and the second one. We combine the signatures in every pair into a signature of highest weight $\widehat\omega_{n}$ or $2\omega_{n}$. Then the $\psi$-image of such signatures have the highest weights $\widehat\omega_{n-1}$ or $2\omega_{n-1}$, therefore there exists a signature $\tau_i$ with the same highest weight.

If $\sigma\in\Sigma''_{D_{n+1}}$ and (or) $\tau_i$ have highest weights in $\{\omega_1,\ldots,\omega_{n-2},\omega_n,2\omega_{n},\widehat\omega_{n-1}\}$ then we use the same arguments.
\end{proof}
The assumption about signatures of highest weight $\omega_{n-1}$ (resp. $\omega_n$) is significant. Indeed, let $\sigma$ have the highest weight $\widehat\omega_n$ and $\psi(\sigma)=\tau_1+\tau_2$, where $\tau_i$ have the highest weight $\omega_n$. Then the above statement does not hold.

\begin{lemma}
\label{odnazv}
The property $(*)$ holds for $\Sigma'_{D_{n+1}}$ and $\Sigma''_{D_{n+1}}$.

\label{theorem5}
\end{lemma}
\begin{proof}
For any signature $\sigma$ of $D_{n+1}$ we denote by $s^{\pm}_{i}(\sigma)$ the exponents of $\sigma$ corresponding to the roots $\varepsilon_i\pm\varepsilon_{n+1}$.
Assume that we have two decompositions of some signature $\sigma$ in $D_{n+1}$:
$$
\sigma=\sigma_1+\ldots+\sigma_k=\sigma'_1+\ldots+\sigma'_l,
$$
where $\sigma_i$ and $\sigma'_i$ are essential signatures for $D_{n+1}$ of highest weights in $$\{\omega_1,\ldots,\omega_{n-1},\omega_{n},2\omega_{n},\widehat\omega_{n}\}$$
or in
$$
\{\omega_1,\ldots,\omega_{n-1},\omega_{n+1},2\omega_{n+1},\widehat\omega_{n}\}.
$$

 One can use arguments similar to those in Lemma \ref{theorem2} by using Corollary \ref{sledlift2} to reduce to the case
$$
\psi(\sigma_1)=\psi(\sigma'_1),\quad\ldots\quad,\quad\psi(\sigma_l)=\psi(\sigma'_l).
$$
Moreover, we may assume that $\psi(\sigma_i)$ has the highest weights in  $\{\omega_1,\ldots,\omega_{n-1},2\omega_{n-1},\widehat\omega_{n-1}\}$ or in $\{\omega_1,\ldots,\omega_{n-2},\omega_n,2\omega_n,\widehat\omega_{n-1}\}$.

Suppose that $\psi(\sigma_1)$ has the highest weight $\omega_{m},m<n-1$, such that among $\psi(\sigma_i)$ there are no  signatures of highest weights $\omega_i,i<m$. By Lemma~\ref{lemma2} the highest weights of both signatures $\sigma_1,\sigma'_1$ are in $\{\omega_m,\omega_{m+1},\widehat\omega_{m+2}\}$. If $\sigma_1=\sigma'_1$ then we are done by induction on $l$. In the remaining cases we provide an algorithm for obtaining the signature $\sigma'_1$ in the left-hand decomposition or the signature $\sigma_1$ in the right-hand decomposition by admissible operations. Then we finish the proof again by induction on $l$.

Denote the highest weight of $\sigma_1$ by
$\lambda$. We arrive at one of the following cases (by symmetry between $\sigma_1$ and $\sigma'_1$ we may suppose that $\lambda\neq\widehat\omega_{m+2}$, because if both $\sigma_1$ and $\sigma'_1$ have the highest weight $\widehat\omega_{m+2}$, then they coincide by Lemma \ref{lemma2}):
\begin{enumerate}
\item $\lambda=\omega_{m}$,
\item $\lambda=\omega_{m+1}$, $s^{+}_{m+1}(\sigma_1)=1$,
\item $\lambda=\omega_{m+1}$, $s^{-}_{m+1}(\sigma_1)=1$.
\end{enumerate}
In the first case one has a signature of highest weight $\lambda$ in the right-hand decomposition, say $\sigma'_2$, such that $s_{i}^{\pm}(\sigma'_2)=0,$ $i=1,\ldots,n,$. Indeed, one must have the signature of highest weight $\lambda$ in the right-hand decomposition and all $s_{i}^{\pm}=0$ for all such signatures, otherwise we would have the contradiction with the choice of the weight $\omega_m$. In the second case one has a signature of highest weight $\lambda$ in the right-hand decomposition, say $\sigma'_2$, such that $s^{+}_{m+1}(\sigma'_2)=1$. Indeed, one must have the signature in the right-hand decomposition with $s^{+}_{m+1}=1$. By Lemma \ref{lemma2} all such signatures have the highest weight $\omega_{m+1}$. In both cases we may swap all exponents of signatures $\sigma'_1$ and $\sigma'_2$ except ones corresponding to roots $\varepsilon_i\pm\varepsilon_{n+1}$  and obtain the signature $\sigma_1$ in the right-hand decomposition. By Lemma \ref{lemma2} it is easy to verify that the signatures obtained by swapping are essential.

In the third case we may assume $s^{-}_{m+1}(\sigma'_1)=s^{+}_{m+2}(\sigma'_1)=1$. Otherwise we would have the first or the second cases. One has the signature of highest weight $\widehat\omega_{m+2}$, say $\sigma_2$, in the left-hand decomposition of $\sigma$ such that $s^{+}_{m+2}(\sigma_2)=1$. We may swap all exponents of signatures $\sigma_1$ and $\sigma_2$ except one corresponding to the root $\varepsilon_{m+2}+\varepsilon_{n+1}$ and obtain the signature
 $\sigma'_1$ in the left-hand decomposition of $\sigma$. By Lemma \ref{lemma2} it is easy to verify that the signatures obtaining by swapping are essential.

 Arguments above show that we may assume that all $\psi(\sigma_i)$ has the highest weight in   $\{\omega_{n-1},\omega_{n},2\omega_{n-1},2\omega_{n},\widehat\omega_{n-1}\}$.  Let, for example, $\psi(\sigma_1)$ have the highest weight $2\omega_n$. By Lemma~\ref{lemma2} we may assume that the highest weights of both signatures $\sigma_1,\sigma'_1$ are in $\{\widehat\omega_{n},2\omega_{n+1}\}$. (In the case of the highest weights of $\sigma_1,\sigma'_1$ in  $\{2\omega_{n},\widehat\omega_{n}\}$ we use similar arguments.) If $\sigma_1=\sigma'_1$ then we are done by induction on $l$. Hence we may suppose that the highest weights of $\sigma_1,\sigma'_1$ are $\widehat\omega_n,2\omega_{n+1}$, respectively, and $s^{\pm}_{i}(\sigma_1)=s^{\pm}_{i}(\sigma'_1)=0$. Since all signatures $\sigma_i,\sigma'_i$ are in $\Sigma''_{D_{n+1}}$ the sets of highest weights of signatures $\{\sigma_i\}$ and $\{\sigma'_i\}$ coincide and one has a signature of highest weight $2\omega_{n+1}$, say $\sigma_2$, in the left-hand decomposition of $\sigma$. One has $s^{+}_n(\sigma_2)\neq2$ since $\psi(\sigma_i)\in\Sigma''_{D_n}$. Then we may swap all exponents of signatures $\sigma_1$ and $\sigma_2$ and obtain the signature $\sigma'_1$ in the left-hand decomposition of $\sigma$. Induction on $l$ finishes the proof. One can use similar arguments in the remaining cases.
\end{proof}

\begin{lemma}
\label{dvezv}
The property $(**)$ holds for $\Sigma'_{D_{n+1}}$ and $\Sigma''_{D_{n+1}}$.
\end{lemma}
\begin{proof}
Similar to Lemma \ref{theorem3}.
\end{proof}
\begin{sled}
The semigroup of essential signatures for $D_{n+1}$ is generated by essential signatures of highest weights in $\{\omega_1,\ldots,\omega_{n+1},\widehat\omega_{n},2\omega_{n},2\omega_{n+1}\}$.
\end{sled}

 Recall (see 4.2) that we denote by ($\sharp'$) the inequalities on the coordinates of a signature which define the semigroup $\Sigma_{D_{n}}$. Let $\sum k'_i\omega_i$ denote the highest weight of a signature for $D_{n}$. Denote by ($\sharp$) the inequalities obtaining from ($\sharp'$) by expressing $k'_i$ via $k_i,s_i^{\pm}$ using the formulas (\ref{defw2}).
\begin{proposition}
\label{conj3dn}
The
semigroup $\Sigma_{D_{n+1}}$ is given by the inequalities $(\sharp)$ and $(\ref{ineq})$.
\end{proposition}
\begin{proof}
Similar to Proposition \ref{conj3bn}.
\end{proof}
\begin{sled}
The semigroup $\Sigma_{D_{n+1}}$ is saturated.
\end{sled}

\subsection{Results}
Finally, we have the numerations on the sets of positive roots for $D_n$ and $B_n$ as follows:
$$
\varepsilon_1-\varepsilon_2,\varepsilon_1+\varepsilon_2,\ldots,\varepsilon_{1}-\varepsilon_n,\ldots,\varepsilon_{n-1}-\varepsilon_{n},\varepsilon_1+\varepsilon_n,\ldots,\varepsilon_{n-1}+\varepsilon_n\quad \textrm{for $D_n$},
$$
$$
\textrm{roots of $D_n$},\varepsilon_1,\ldots,\varepsilon_n\quad \textrm{for $B_n$}.
$$
Also we have the monomial order on the set of signatures. We compare two signatures of $D_n$ of the same highest weight as follows (we move on to the next step if on the previous steps the tuples of exponents of the signatures coincide):
\begin{enumerate}
\item compare the tuples of exponents corresponding to the roots
 $\varepsilon_1+\varepsilon_n,\ldots,\varepsilon_{n-1}+\varepsilon_n$ by using the degree lexicographic order,

 \item compare the tuples of exponents corresponding to the roots $\varepsilon_1-\varepsilon_n,\ldots,\varepsilon_{n-1}-\varepsilon_n$ by the degree lexicographic order,
     \item compare the tuples of exponents corresponding to the roots $\varepsilon_1+\varepsilon_{n-1},\ldots,\varepsilon_{n-2}+\varepsilon_{n-1}$ by the degree lexicographic order,
         \item compare the tuples of exponents corresponding to the roots $\varepsilon_1-\varepsilon_{n-1},\ldots,\varepsilon_{n-2}-\varepsilon_{n-1}$ by the degree lexicographic order,
 \item[]\ldots\ldots
 \item compare the exponents corresponding to the root $\varepsilon_1+\varepsilon_2$,
     \item compare the exponents corresponding to the root $\varepsilon_1-\varepsilon_2$.
\end{enumerate}
We compare two signatures of $B_n$ of the same highest weight as follows (we move on to the 2nd step if on the 1st step the tuples of exponents of the signatures coincide):
\begin{enumerate}
\item compare the tuples of exponents corresponding to the roots
 $\varepsilon_1,\ldots,\varepsilon_n$ by using the degree lexicographic order,
 \item compare the tuples of exponents corresponding to the long roots by the above rule for $D_n$.
\end{enumerate}
The inductive argument in subsections 4.2 and 4.3 proves our main result:
\begin{theor}
 \label{main}
 The following statements are true for $D_n$, $B_n$ ($n\geq2$):

For the numeration of positive roots and the monomial order on signatures described above, the semigroup $\Sigma$ is saturated and generated by essential signatures of highest weights in the set $$\{\omega_1,\ldots,\omega_n,2\omega_{n-1},2\omega_n,\omega_{n-1}+\omega_n\}\quad\textrm{for $D_n$}$$
and in the set
  $$\{\omega_1,\ldots,\omega_n,2\omega_{n}\}\quad\textrm{for $B_n$}.$$

\end{theor}
Let $p_{i,j}^{\pm}\geq0$, $i<j$, be the exponent of a signature of $D_n$ corresponding to the root $\varepsilon_i\pm\varepsilon_j$, respectively, $\sum k_i\omega_i$ be a dominant weight. One can easily deduce the inequalities on $p_{i,j}^{\pm},k_i$ which define the semigroup $\Sigma_{D_n}$ by using Proposition \ref{conj3dn} and inequalities defining $\Sigma_{D_2}$.
\begin{theor}
\label{ineqdn}
The semigroup $\Sigma_{D_n}$ is given by the inequalities:
\begin{enumerate}
\item $p_{i,j}^{+}+\sum_{k=j+1}^{n} (p_{i,k}^{+}+ p_{i,k}^{-}-p_{i+1,k}^{+}-p_{i+1,k}^{-})\leq k_i,\\ i\in\{1,\ldots,n-2\}, j\in\{i+2,\ldots,n\},$

\item $p_{i+1,j}^{-}+\sum_{k=j}^{n} (p_{i,k}^{+}+ p_{i,k}^{-}-p_{i+1,k}^{+}-p_{i+1,k}^{-})\leq k_i,\\ i\in\{1,\ldots,n-2\}, j\in\{i+2,\ldots,n\},$

\item $p_{i,i+1}^{-}+\sum_{k=i+2}^{n} (p_{i,k}^{+}+ p_{i,k}^{-}-p_{i+1,k}^{+}-p_{i+1,k}^{-})\leq k_i,\\i\in\{1,\ldots,n-1\},$
\item $p_{i,i+2}^{-}+p_{i,i+2}^{+}+p_{i+1,i+2}^{-}+\sum_{k=i+3}^{n} (p_{i,k}^{+}+ p_{i,k}^{-}-p_{i+2,k}^{+}-p_{i+2,k}^{-})\leq k_i+k_{i+1},\\i\in\{1,\ldots,n-2\},$
\item $p_{i,i+1}^{+}+\sum_{k=i+2}^{n} (p_{i,k}^{+}+ p_{i,k}^{-}+p_{i+1,k}^{+}+p_{i+1,k}^{-})\leq k_i+k_{n-1}+k_n+2\sum_{j=i+1}^{n-2}k_j,\\i\in\{1,\ldots,n-2\},$
\item $p_{n-1,n}^{+}\leq k_n.$
    \end{enumerate}
\end{theor}
 Let $s_i\geq0$ be the exponent of a signature of $B_n$ corresponding to the root $\varepsilon_i$. Let $p_{i,i}^{\pm}:=\frac{1}{2}s_i$. The inequalities which define the semigroup $\Sigma_{B_n}$ can be easily deduced from Proposition \ref{conj3bn} and the previous theorem.
 \begin{theor}
 \label{ineqbn}
The semigroup $\Sigma_{B_n}$ is given by the inequalities:
\begin{enumerate}
\item $s_i-s_{i+1}+p_{i,j}^{+}+\sum_{k=j+1}^{n} (p_{i,k}^{+}+ p_{i,k}^{-}-p_{i+1,k}^{+}-p_{i+1,k}^{-})\leq k_i,$\\ $i\in\{1,\ldots,n-2\}, j\in\{i+2,\ldots,n\},$

\item $s_i-s_{i+1}+p_{i+1,j}^{-}+\sum_{k=j}^{n} (p_{i,k}^{+}+ p_{i,k}^{-}-p_{i+1,k}^{+}-p_{i+1,k}^{-})\leq k_i,\\ i\in\{1,\ldots,n-2\}, j\in\{i+2,\ldots,n\},$

\item $s_i-s_{i+1}+p_{i,i+1}^{-}+\sum_{k=i+2}^{n} (p_{i,k}^{+}+ p_{i,k}^{-}-p_{i+1,k}^{+}-p_{i+1,k}^{-})\leq k_i,\\i\in\{1,\ldots,n-1\},$
\item $ s_i+p_{i+1,i+2}^{-}+\sum_{k=i+2}^{n} (p_{i,k}^{+}+ p_{i,k}^{-}-p_{i+2,k}^{+}-p_{i+2,k}^{-})\leq k_i+k_{i+1},$\\
    $i\in\{1,\ldots,n-2\},$
\item $ s_i-p_{i,i+1}^{-}+\sum_{k=i+1}^{n} (p_{i,k}^{+}+ p_{i,k}^{-}+p_{i+1,k}^{+}+p_{i+1,k}^{-})\leq k_i+k_n+2\sum_{j=i+1}^{n-1}k_j,\\i\in\{1,\ldots,n-1\},$
\item $s_i\leq k_i,$ $i\in\{1,\ldots,n\}.$
    \end{enumerate}
\end{theor}


\begin{thebibliography}{FFL2}
\bibitem{[T]}
T. Backhaus, D. Kus, \emph{The PBW filtration and convex polytopes in type B}, Journal of Pure and Applied Algebra, to appear.

\bibitem{[FFL1]}
E. Feigin, G. Fourier, P. Littelmann, \emph{PBW filtration and bases for irreducible modules in type $A_n$}, Transformation Groups \textbf{165} (2011), no. 1, 71-89.
\bibitem{[FFL2]}
E. Feigin, G. Fourier, P. Littelmann, \emph{PBW filtration and bases for symplectic Lie algebras}, Int. Math. Res. Not. 2011, no. 24, 5760-5784.
\bibitem{[G]}
A.A. Gornitskii, \emph{Essential signatures and canonical bases of irreducible representations of the group $G_2$}, Mat. Zametki \textbf{97} (2015), 35-47 (in Russian); English translation: Mathematical Notes \textbf{97} (2015), 30-41.
\bibitem{[G1]}
A.A. Gornitskii, \emph{Essential signatures and canonical bases of irreducible representations of $D_4$},
 arXiv:1507.07498.
\bibitem{[X]}
 J. E. Humphreys, \emph{Linear Algebraic Groups}, Springer, New York, 1977.

 \bibitem{[L]}
 P. Littelmann, \emph{Cones, crystals, and patterns}, Transformation Groups \textbf{3} (1998), no. 2, 145-179.
 \bibitem{[Li]}
 P. Littelmann, X. Fang, G. Fourier, \emph{Essential bases and toric degenerations arising from generating
sequences}, Advances in Mathematics \textbf{312} (2017), 107-149.
\bibitem{[Lu]}
G. Lusztig, \emph{Canonical bases arising from quantized enveloping algebras}, J. Amer. Math. Soc. \textbf{3} (1990), no. 2, 447-498.
\bibitem{[Ì]}
I. Makhlin, \emph{FFLV-type monomial bases for type B},  arXiv:1610.07984.
\bibitem{[Ï]}
 V. L. Popov, \emph{Contractions of the actions of reductive algebraic groups}, Mat. Sb. \textbf{130} (1986), no. 3, 310-334 (in Russian); English translation: Math. USSR-Sbornik \textbf{58} (1987), no. 2, 311-335.
 \bibitem{[V-Ï]}
 V. L. Popov, E. B. Vinberg, \emph{On a class of quasihomogeneous affine varieties}, Izv. Akad. Nauk SSSR Ser. Mat. \textbf{36} (1972), no. 4, 749-764 (in Russian); English translation: Math. USSR-Izvestiya \textbf{6} (1972), no. 4, 743-758.
\bibitem{[V]}
 E. B. Vinberg, \emph{On some canonical bases of representation spaces of simple Lie algebras}, Conference Talk, Bielefeld, 2005.
 \bibitem{[ÂÎ]}
 E. B. Vinberg, A. L. Onishchik, \emph{Seminar on Lie Groups and Algebraic Groups}, Nauka, Moscow, 1988 (in Russian); English translation: \emph{Lie Groups and Algebraic Groups}, Springer, Berlin, 1990.
\bibitem{[Z]}
D. P. Zhelobenko, \emph{Compact Lie Groups and Their Representations}, Nauka, Moscow, 1970 (in Russian); English translation: American Mathematical Society, Providence, 1973.
\vspace{5pt}
\end{thebibliography}
\end{document}